\documentclass[11pt, reqno]{amsart}
\usepackage{amssymb}
\usepackage{amsmath}
\usepackage{amsthm}
\usepackage{braket}
\usepackage{pdfpages}
\usepackage{graphicx}
\usepackage{caption}
\usepackage{subcaption}
\usepackage{mathrsfs} 
\usepackage{tikz-cd} 
\usepackage{bm}
\usepackage{enumitem}
\usepackage{mathtools}
\usepackage{pgfplots}
\usepackage{import}
\usepackage{xifthen}
\usepackage{transparent}
\usepackage{mathtools}
\usepackage{tabularx}
\usepackage{hhline}
\usepackage{scrextend}

\usepackage[utf8]{inputenc}
\usepackage[english]{babel}

\usepackage{hyperref}

\usepackage{lmodern,babel,adjustbox,booktabs,multirow}

\numberwithin{equation}{section}
\theoremstyle{plain}
\newtheorem{theorem}{Theorem}[section]

\theoremstyle{theorem}
\newtheorem{prop}[theorem]{Proposition}
\newtheorem{lem}[theorem]{Lemma}
\newtheorem{cor}[theorem]{Corollary}

\theoremstyle{definition}
\newtheorem{defn}[theorem]{Definition}

\newtheorem*{remark}{Remark}
\newtheorem*{conv}{Convention}
\newtheorem*{example}{Example}

\newcommand{\R}{\mathbb{R}}
\newcommand{\C}{\mathbb{C}}
\newcommand{\Proj}{\mathbb{P}}

\newcommand{\Z}{\mathbb{Z}}
\newcommand{\Hyp}{\mathbb{H}}
\newcommand{\N}{\mathbb{N}}

\newcommand{\E}{\mathscr{E}}

\DeclareMathOperator{\dist}{d}

\DeclareMathOperator{\Rat}{Rat}

\DeclareMathOperator{\PSL}{PSL}
\DeclareMathOperator{\Isom}{Isom}
\DeclareMathOperator{\Conf}{Conf}
\DeclareMathOperator{\Jac}{Jac}
\DeclareMathOperator{\EL}{\mathcal{L}}
\DeclareMathOperator{\EW}{\mathcal{W}}
\DeclareMathOperator{\Res}{Res}

\DeclareMathOperator{\Mod}{mod}

\DeclareMathOperator{\capacity}{cap}
\newcommand{\Derivative}{D}
\DeclareMathOperator{\hull}{hull}
\DeclareMathOperator{\Hull}{Hull}
\newcommand{\powerset}{\raisebox{.15\baselineskip}{\Large\ensuremath{\wp}}}

\numberwithin{figure}{section}

\title{Limits of rational maps, $\R$-trees and barycentric extension }
\author{Yusheng Luo}
\address{Institute of Mathematical Science, Dept. of Mathematics, Stony Brook University, Stony Brook, NY 11794 USA}
\email{yusheng.s.luo@gmail.com}
\date{\today}
\begin{document}

\begin{abstract}
In this paper, we show that one can naturally associate a limiting dynamical system $F: T\longrightarrow T$ on an $\R$-tree to any degenerating sequence of rational maps $f_n: \hat\C \longrightarrow \hat\C$ of fixed degree.
The construction of $F$ is in $2$ steps: first, we use barycentric extension to get $\E f_n : \Hyp^3 \longrightarrow  \Hyp^3$; second, we take appropriate limit on rescalings of hyperbolic space.
An important ingredient we prove here is that the Lipschitz constant depends only on the degree of the rational map.
We show that the dynamics of $F$ records the limiting length spectra of the sequence $f_n$.
\end{abstract}

\maketitle

\setcounter{tocdepth}{1}
\tableofcontents

\section{Introduction}
A rich theory is developed in the study of the dynamics of rational maps $f: \hat\C \longrightarrow \hat\C$ with $d = \deg(f) \geq 2$.
The space $\Rat_d(\C)$ of all rational maps of degree $d$ is not compact.
Hence it is useful to construct limiting dynamical systems as $f_n \to \infty$ in $\Rat_d(\C)$.

In this paper, we show that to any sequence $f_n \to \infty$, one can naturally associate a degree $d$ branched covering on an $\R$-tree $F: T\longrightarrow T$.
The map $F$ is constructed in two steps.
\begin{enumerate}
\item By applying the barycentric extension, we extend each rational map $f_n: \hat\C \longrightarrow \hat\C$ to a map $\E f_n: \Hyp^3\longrightarrow \Hyp^3$ (see \S \ref{BE}, or \cite{DE86,CP11}).
\item Using the standard idea that $\Hyp^3$ takes on the appearance of an $\R$-tree under rescaling (cf. \cite{Gromov81}), we construct the branched covering by taking limits of barycentric extensions $\E f_n$ with appropriate rescalings.
\end{enumerate}
The following equicontinuity theorem plays a crucial role.
\begin{theorem}\label{RationalLip}
For any rational map $f: \hat\C \longrightarrow \hat\C$ of degree $d$, the norm of the derivative of its barycentric extension $\E f: \Hyp^3 \longrightarrow \Hyp^3$ satisfies
$$
\sup_{x\in \Hyp^3} \| \Derivative \E f (x)\| \leq C \deg(f).
$$
Here the norm is computed with respect to the hyperbolic metric, and $C$ is a universal constant.
\end{theorem}

\begin{remark}
The theorem in particular says that the Lipschitz constant of $\E f$ is $C\deg(f)$, i.e., $\forall x, y\in \Hyp^3$,
$$
\dist_{\Hyp^3}(\E f(x), \E f(y)) \leq C \deg(f) \dist_{\Hyp^3}(x, y).
$$

In dimension $1$, if $f: S^1\longrightarrow S^1$ is the restriction of a Blaschke product, then the barycentric extension $\E f: \Hyp^2\cong \Delta \longrightarrow \Hyp^2$
satisfies $\E f = f$.
Hence, the Schwarz lemma implies that $\E f$ is a distance non-increasing map with respect to hyperbolic metric, i.e., uniformly $1$-Lipschitz.
Theorem \ref{RationalLip} is a replacement of this geometric bound in higher dimensions.

Our proof shows that one can take the constant $C=\frac{27}{2\log 3}\approx 12.3$.
On the other hand, the barycentric extension of $z^d$ has Lipschitz constant $d$ on the geodesic connecting $0, \infty\in S^2$.
Hence $C \geq 1$.
We conjecture that $C=1$ gives the sharp bound for the inequality.

There is a more general statement for quasiregular maps and in higher dimensions. See Theorem \ref{Lip} in Appendix \ref{HD} for details.
\end{remark}

\subsection*{Geometric limits as branched coverings on an $\R$-tree}
Back to the study of $f_n \to \infty$,
Theorem \ref{RationalLip} shows $\E f_n$ is equicontinuous. 
However,
\begin{enumerate}
\item $\E f_n$ is not uniformly bounded in general: the image of $\bm 0\in \Hyp^3$ may escape to infinity (see Proposition \ref{bdd});
\item even if $\E f_n$ is uniformly bounded, the degree will definitely drop: some preimages of $\bm{0}$ will escape to infinity (see Proposition \ref{LossOfMass}). 
\end{enumerate}
It is thus natural to consider limits of rescalings of $\Hyp^3$.

Given a sequence of rescalings $r_n \to \infty$, we will construct the {\em asymptotic cone} $({^r\Hyp^3}, x^0, \dist)$ for the sequence of pointed metric spaces
$(\Hyp^3, \bm{0}, \dist_{\Hyp^3}/r_n)$.

This construction uses a {\em non-principal ultrafilter} $\omega$ which will be fixed once for all.
The ultrafilter $\omega$ allows any sequences in compact spaces to converge automatically without the need to pass to a subsequence.
We will denote the limit with respect to the ultrafilter by $\lim_\omega$(see \S \ref{sec:ulac}).

It is well-known that the asymptotic cone $({^r\Hyp^3}, x^0, \dist)$ is an $\R$-tree \cite{Roe03}. 
This means that any two points $x, y\in {^r\Hyp^3}$ can be joined by a unique arc $[x,y]\subseteq {^r\Hyp^3}$ which is isometric to an interval of $\R$.

The appropriate scale we consider here is
$$
r_n := \max_{y\in \E f_n ^{-1}(\bm 0)} \dist_{\Hyp^3}(\bm 0, y),
$$
as it brings all the preimages of $\bm 0$ in view.
We show that
\begin{theorem}\label{DGL}
Let $f_n \to \infty$ in $\Rat_d(\C)$. Let
$$
r_n:= \max_{y\in \E f_n^{-1}(\bm{0})} \dist_{\Hyp^3}(\bm{0}, y),
$$ 
and ${^r\Hyp^3}$ be the asymptotic cone of $\Hyp^3$ with rescaling $r_n$. 
Then we have a limiting map
$$
F = \lim_\omega \E f_n : {^r\Hyp^3} \longrightarrow {^r\Hyp^3},
$$
and it is a degree $d$ branched covering of the $\R$-tree ${^r\Hyp^3}$.
\end{theorem}

\begin{remark}
Let $s_n \to\infty$ be a rescaling with $\dist_{\Hyp^3}(\bm{0}, \E f_n(\bm{0})) \leq K s_n$ for some $K >0$.
Then our construction gives a limiting map $F': {^s\Hyp^3} \longrightarrow {^s\Hyp^3}$ (see Proposition \ref{GLTree}).
\begin{itemize}
\item If $\lim_\omega s_n / r_n = M \in (0,\infty)$, the two limiting maps are conjugate to each other. More precisely, there exists a homeomorphism $\Phi: {^r\Hyp^3} \rightarrow {^s\Hyp^3}$, which is a dilation by $M$, conjugating $F$ and $F'$.
\item If $\lim_\omega s_n / r_n = 0$, the degree of the limiting map will drop.
\item If $\lim_\omega s_n / r_n = \infty$, then $x^0\in {^r\Hyp^3}$ is totally invariant, resulting in a less interesting dynamical system.
\end{itemize}

We have a version for sequences $[f_n] \to \infty$ in the moduli space of rational maps $M_d = \Rat_d(\C) / \PSL_2(\C)$ (see Theorem \ref{DGLC}).

We also show how to extract standard geometric limits of $\E f_n$ from the ultralimit $F: {^r\Hyp^3} \longrightarrow {^r\Hyp^3}$ (see Theorem \ref{thm:gl}).
\end{remark}

Although the barycentric extension does not behave well under composition, we show that the limiting map is dynamically natural (cf. \cite[Theorem 1.5]{DeM08}).
\begin{theorem}[Dynamical naturality]\label{DN}
Let $f_n\to\infty$ in $\Rat_d(\C)$, with limiting map $F: {^r\Hyp^3} \longrightarrow {^r\Hyp^3}$. Then for any $N \in \N$,
$$
F^N = \lim_\omega \E (f_n^N).
$$
\end{theorem}

\subsection*{Length spectra and translation lengths}
A ray $\alpha$ in the $\R$-tree $T$ is a subtree isometric to $[0,\infty) \subseteq \R$.
Two rays $\alpha_1, \alpha_2$ are {\em equivalent}, denoted by $\alpha_1\sim \alpha_2$, if $\alpha_1 \cap \alpha_2$ is still a ray.
The collection $\epsilon(T)$ of all equivalence classes of rays forms the set of {\em ends} of $T$.
We will use $\alpha$ to denote both a ray and the end it represents.

We say a sequence of points $x_i$ converges to an end $\alpha$, denoted by $x_i\to \alpha$, if for any ray $\beta \sim \alpha$, $x_i \in \beta$ for all sufficiently large $i$.

We define the {\em translation length} of an end $\alpha$ by
$$
L(\alpha, F) = \lim_{x_i \to \alpha} \dist(x_i, x^0) - \dist(F (x_i), x^0)
$$
for any $x_i\in T$ with $x_i \to\alpha$.
We show that the translation length is well-defined for the limiting map $F: {^r\Hyp^3} \longrightarrow {^r\Hyp^3}$ (see \S \ref{PL}).

If $\mathcal{C} = \{\alpha_1, ..., \alpha_q\}$ is a cycle of periodic ends, then the {\em translation length} of the periodic cycle $\mathcal{C}$ is
$$
L(\mathcal{C}, F) = \sum_{i=1}^q L(\alpha_i, F).
$$

Let $f \in \Rat_d(\C)$, and $C= \{z_{1},...,z_{q}\} \subseteq \hat\C$ be a periodic cycle of $f$ of period $q$. 
We define the {\em length} of the periodic cycle $C$ by
$$
L(C, f) := \log |(f^q)'(z_{1})| = \sum_{i=1}^q \log |f'(z_{i})|.
$$

For the asymptotic cone ${^r\Hyp^3}$,
{\em any} sequence $(z_n) \in \hat\C$ determines an end $\alpha \in \epsilon({^r\Hyp^3})$ (see \S \ref{sec:ulac}).
We denote it by $z_n \twoheadrightarrow_\omega \alpha$.

More generally, if $C_n=\{z_{1,n},..., z_{q,n}\} \subseteq \hat\C$ and $\mathcal{C} = \{\alpha_1, ..., \alpha_q\} \subseteq \epsilon({^r\Hyp^3})$ with $z_{i,n}\twoheadrightarrow_\omega \alpha_i$, then we denote it by $C_n \twoheadrightarrow_\omega \mathcal{C}$.

We show that we can recover the limiting length spectra from the translation lengths for the limiting map $F$ (cf.  \cite[Theorem 1.8]{DeM08} and \cite[Theorem 1.4]{McM09}).
\begin{theorem}\label{TL}
Let $f_n\to\infty$ in $\Rat_d(\C)$, with limiting map $F: {^r\Hyp^3} \longrightarrow {^r\Hyp^3}$.
Let $C_n \subseteq \hat\C$ be a sequence of periodic cycles of $f_n$ of period $q$, with $C_n \twoheadrightarrow_\omega \mathcal{C}$. Then $\mathcal{C}$ is periodic with period $q$, and
$$
L(\mathcal{C}, F) = \lim_\omega \frac{L(C_n, f_n)}{r_n}.
$$
\end{theorem}

\subsection*{Discussion on related work}
Limiting maps on trees are constructed for sequences of polynomials in \cite{DeM08} and sequences of Blaschke products in \cite{McM09}.
Theorem \ref{DGL} generalizes these two constructions, in the sense that the corresponding limiting trees for polynomials or Blaschke products arise as invariant subtrees of ${^r\Hyp^3}$ under the limiting branched covering $F$.
Theorem \ref{TL} is a direct generalization of \cite[Theorem 1.8]{DeM08} and \cite[Theorem 1.4]{McM09}.
For both polynomials and Blaschke products, the limiting trees are simplicial.
This is not the case for limits of general sequences of rational maps (see, for example, \cite[\S 7]{L21}).

Dynamics on a finite tree of spheres appears naturally as a limit of certain sequences of rational maps, and is systematically studied in \cite{Arfeux16, Arfeux17}.
Such constructions have already appeared implicitly in Shishikura's proof of the famous Fatou-Shishikura inequality \cite{Shishikura87, Shishikura89}.
These tree maps are related to the notion of {\em rescaling limits} introduced by Kiwi and can be formulated in terms of Berkovich dynamics \cite{Kiwi15,BR10}.
In the sequel \cite{L19}, we will compare our construction of branched coverings of the $\R$-tree using the barycentric extension to the Berkovich dynamics of the complexified Robinson's field. 
We will show that the two approaches are equivalent.
We will also apply the relation in Theorem \ref{TL} to classify those hyperbolic components admitting a degenerating sequence where all the multipliers stay bounded.

There are various other compactifications of the moduli space of rational maps or its variants \cite{Milnor93, Sil98, DeM07, DeM08, McM09b, McM09}.
Unlike the compactification using the geometric invariant theory in \cite{Sil98}, Theorem \ref{DN} asserts that iteration extends continuously to the boundary of our compactification using $\R$-trees (cf. the compactification of the moduli space of quadratic rational maps in \cite{DeM07} and the moduli space of polynomials in \cite{DeM08}).
However, we do not know the topology of our compactification, and it would be interesting to figure out explicit connections with others.
To illustrate some of the complexities of the boundary of our compactification, for degree $2$ Blaschke products, it is shown in \cite[Theorem 1.7]{McM09} that the boundary is parameterized by the circle $\R/\Z$ with its rational points blown up to intervals.
Very little is known of the boundary even for degree $3$ Blaschke products.
Many interesting examples of limiting branched coverings on $\R$-trees are constructed in \cite{L19, L21}. 
One would like to know if there is a characterization on which branched coverings can arise as the limit.

Our constructions of branched coverings on $\R$-trees run parallel with the corresponding development for Kleiniain groups, and fit into the well-known {\em Sullivan's dictionary}.
An important ingredient of Thurston's theory on surface automorphisms is his construction of an equivariant compactification of the Teichm\"uller space \cite{T88,FLP79}.
Isometric actions on $\R$-trees provide such compactifications for various spaces of geometric structures.
In \cite{MS84},  Morgan and Shalen showed how to assign an isometric action on a $\R$-tree to a degenerating sequence of representations.
This shed new light and generalized parts of Thurston's Geometrization Theorem for $3$-manifolds.
Bestvina and Paulin gave a new geometric perspective on this theory in \cite{B88,P88}.
$\R$-trees were also used by Otal to give a proof of Thurston's Double Limit Theorem and the Hyperbolization Theorem for 3-manifolds that fiber over the circle \cite{O96}.
The use of asymptotic cone and its connection with Gromov-Hausdorff limit are explained in \cite{O96,KapovichLeeb95,C91}.
The geometric limits constructed in Theorem \ref{DGL} and Theorem \ref{DGLC} 
give analogous compactifications as in \cite{MS84,B88,P88}.
The limiting ratios of lengths of marked geodesics for a degenerating sequence of Riemann surfaces or Kleinian groups are naturally recorded in the limiting action on the $\R$-tree \cite{MS84,O96}.
Theorem \ref{TL} gives the analogous result in complex dynamics.


The barycentric extension method has been used in various geometric and dynamics settings.
In \cite{BCG95,BCG96}, Besson, Coutois and Gallot used the explicit bound on the Jacobian of the barycentric extension of the Patterson-Sullivan measure to prove rigidity results on negatively curved Riemannian manifold.
Other applications can be found in \cite{EMS02,S07}.
The barycentric method is also frequently used in studying the measure of maximal entropy in complex dynamics \cite{DeM07, Jacobs18}.

\subsection*{Supplement: the barycentric extension of $z^2$}
We conclude our discussion by studying the barycentric extension of the map $f(z) = z^2$.
The exact formula for the extension is hard to compute, even in this simple case.


Let $(r, \theta, h)$ be the cylindrical coordinates for the hyperbolic $3$-space $\Hyp^3$ with respect to the hyperbolic metric.
In Appendix \ref{BP}, we show that:

{
\renewcommand{\thetheorem}{\ref{Degree2Coordinate}}
\begin{theorem}
Let $f(z) = z^2$.
In the cylindrical coordinate, there exists a real analytic function $\delta: [0,\infty) \longrightarrow \R$ with 
\begin{itemize}
\item $\delta(0) = 0$,
\item $\delta(r) > 0$ when $r>0$,
\item $\lim_{r\to\infty}\delta(r)= 0$,
\end{itemize}
such that the barycentric extension $\E f$ is given by
$$
\E f(r, \theta, h) = (\log (\cosh(r)) - \delta(r), 2\theta, 2h).
$$

\end{theorem}
\addtocounter{theorem}{-1}
}

Since in $(r,\theta)$ coordinate, $z^2$ has the form
$$
(r,\theta) \mapsto (\log (\cosh(r)), 2\theta),
$$
Theorem \ref{Degree2Coordinate} implies the restriction of $\E f$ on the equatorial hyperbolic plane $\Hyp^2_0$ defined by $h=0$ is not $z^2$. This gives a counterexample to a conjecture of Petersen \cite{CP11}.

Since $\Isom(\Hyp^3) \times \Isom(\Hyp^3)$ acts transitively on $\Rat_2(\C)$, by naturality, $\E f|_{\Hyp^2_0} \neq f$ for any degree $2$ Blaschke product $f$.
It seems unlikely that the restriction $\E f|_{\Hyp^2_0}$ of a Blaschke product $f$ would ever agree with $f$.



\subsection*{The structure of the paper}

In \S \ref{BE} and \S \ref{sec:ELW}, we review the construction of the barycentric extension and the theory of extremal lengths. 
Using the classical length-area principle, the proof the Theorem \ref{RationalLip} is given in \S \ref{pf}.

Theorem \ref{RationalLip} allows us to construct various limits for barycentric extensions of a sequence of rational maps. 
We study the limits of barycentric extensions on $\Hyp^3$ in \S \ref{CC} and the limits on rescalings of $\Hyp^3$ in \S \ref{gl}.
Some backgrounds on convergence of annuli and $\R$-trees are provided in \S \ref{sec:CA} and \S \ref{BCR}.
 We study the translation lengths and length spectra in \S \ref{PL}.

Finally, the barycentric extension of $z^2$ is studied in Appendix \ref{BP}, and a higher dimensional analogue of Theorem \ref{RationalLip} is proved in Appendix \ref{HD}.

\subsection*{Acknowledgments.}
The author thanks C. T. McMullen for his advice and helpful discussion on this problem.
The author gratefully thanks the anonymous referee for valuable comments and suggestions.

\section{The barycentric extension}\label{BE}
In this section, we will briefly review the theory of barycentric extensions.
The theory is extensively studied for circle homeomorphisms in \cite{DE86}.
The construction can be easily generalized to any continuous maps on the unit sphere $S^{n}$ (see \cite{McM96,CP11}).
We will then use the implicit differentiation to compute its derivatives.

\subsection*{Definition of barycentric extensions}
Throughout the paper, we identify $\Hyp^{n+1} \cong B(\bm 0,1) \subseteq \R^{n+1}$ as the unit ball centered at the origin $\bm 0 \in \R^{n+1}$.

Let $\mu$ be a probability measure on the unit sphere $S^{n}\subseteq \R^{n+1}$. We say $\mu$ is {\em balanced} if
$$
\int_{S^{n}} \zeta d\mu(\zeta) = \vec 0 \in \R^{n+1}. \footnote{This integral has a natural interpretation of a vector in the tangent space $T_{\bm 0} \Hyp^{n+1}$ (see \cite{DE86} for details), thus we use the usual vector notation $\vec 0$ instead of $\bm 0$.}
$$

We say $\mu$ is {\em balanced at a point} $y\in \Hyp^{n+1}$ if the push forward measure $M_* \mu$ is balanced, where $M \in\Isom \Hyp^{n+1}$ is an isometry sending $y$ to $\bm 0$.
Since the balanced condition is invariant under rotation, it follows that this definition does not depend on the choice of $M$.

If $\mu$ is a probability measure on $S^{n}$ with no atoms, then there is a unique point $\beta(\mu) \in \Hyp^{n+1}$, called the {\em barycenter} of $\mu$, for which the measure is balanced at $\beta(\mu)$ (see \cite{DE86,H06} or \cite{CP11}). 

Let $\mu_{S^{n}}$ be the probability measure coming from the spherical metric on $S^{n}$.
We say a map $f$ is {\em admissible} if $f_*\mu_{S^{n}}$ has no atoms.

Let $f : S^{n} \longrightarrow S^{n}$ be an admissible continuous map and $x\in \Hyp^{n+1}$.
Let $M\in\Isom \Hyp^{n+1}$ be an isometry sending $\bm {0}$ to $x$.
The {\em barycentric extension} is the real analytic map
$$
\E f:\Hyp^{n+1}\longrightarrow \Hyp^{n+1}
$$
that sends $x$ to the barycenter of the measure $(f\circ M)_*\mu_{S^{n}}$.
Since the measure $\mu_{S^n}$ is invariant under rotation, it follows that this definition does not depend on the choice of $M$.

One of the desired properties of the barycentric extension is that it is {\em conformally natural}.
This means that if $M_1, M_2 \in \Isom\Hyp^{n+1}$, then
$$
M_1 \circ \E f \circ M_2 = \E (M_1 \circ f \circ M_2),
$$
where we have identified $\Isom\Hyp^{n+1}$ with the group $\Conf S^{n}$ of conformal maps on $S^{n} = \partial \Hyp^{n+1}$.

\subsection*{Implicit formula for barycentric extensions}
Given a point $x \in \Hyp^{n+1} \cong B(\bm 0,1) \subseteq \R^{n+1}$, we define $M_x: \Hyp^{n+1} \longrightarrow \Hyp^{n+1}$ by
$$
M_{x} (z) = \frac{z(1-\|x\|^2) + x (1+\|z\|^2 + 2 \langle x,z\rangle)}{1+\|x\|^2\|z\|^2 + 2\langle x, z\rangle}.
$$
Then $M_x$ is an isometry on $\Hyp^{n+1}$ sending $\bm {0}$ to $x$, with $M_{x} ^{-1} = M_{-x}$.

Note that $M_x$ extends to a map $M_x: S^n\longrightarrow S^n$. 
An easy computation shows that the Jacobian at $\zeta \in S^{n}$ is
$$
\text{Jac} M_{x} (\zeta) = (\frac{1-\| x\|^2}{\| \zeta +  x\|^2})^{n}.
$$

Let $F: \Hyp^{n+1} \times \Hyp^{n+1} \longrightarrow \R^{n+1}$ be the function
\begin{align*}
F(x, y) &= \int _{S^{n}} M_{y} ^ {-1} ( f (M_{x}( \zeta))) d\mu_{S^{n}}( \zeta)\\
&= \int _{S^{n}} M_{-y} (f (\zeta)) d(M_{x})_*\mu_{S^{n}}(\zeta) \\
&= \int _{S^{n}} M_{-y} (f (\zeta)) (\frac{1-\|x\|^2}{\|\zeta - x\|^2})^{n} d\mu_{S^{n}}(\zeta).
\end{align*}

With the above notations, the barycentric extension of $f$ is the unique solution of the implicit formula: 
$$
F(x, \E f (x)) = \vec 0 \in \R^{n+1}.
$$

Therefore, implicit differentiation gives:
$$
\Derivative \E f (x) = -F_{y} ^{-1} (x,  \E f(x)) F_{x} (x,  \E f(x)).
$$

Since $\Isom\Hyp^{n+1} \times \Isom\Hyp^{n+1}$ acts transitively on pairs of points in $\Hyp^{n+1}$, by naturality of the barycentric extension, it suffices to compute the derivative at the origin $\bm 0$ under the assumption $\E f (\bm 0) = \bm 0$, i.e. 
$$
\int _{S^{n}}  f(\zeta) d\mu_{S^{n}}(\zeta) = \vec 0.
$$
For this, we have

\begin{prop}\label{prop:cd}
Assume that $\E f (\bm 0) = \bm 0$. Then for all $\vec v\in \R^{n+1}$,
$$
F_{y} (\bm 0, \bm 0) (\vec v) = -2 \vec v + 2 \int _{S^{n}}  \langle \vec v, f(\zeta)\rangle f(\zeta) d\mu_{S^{n}}(\zeta),
$$
and
$$
F_{x} (\bm 0, \bm 0) (\vec v) = 2n \int _{S^{n}} \langle \vec v, \zeta\rangle f(\zeta) d\mu_{S^{n}}(\zeta).
$$
\end{prop}

\begin{proof}
For the first equality, we have
\begin{align*}
F_{y} (\bm 0, \bm 0) (\vec v) &=\lim _{t\to 0} \frac{1}{t} \cdot F(\bm 0, t\vec v)\\
&= \lim_{t\to 0} \int_{S^{n}} \frac{1}{t} \cdot  M_{-t\vec v}(f(\zeta)) d\mu_{S^{n}}(\zeta)\\
&= \lim_{t\to 0} \int_{S^{n}}  \frac{1}{t} \cdot \frac{f(\zeta)(1-t^2\|\vec v\|^2)-t\vec v(1+\|f(\zeta)\|^2 - 2t\langle \vec v, f(\zeta)\rangle )}{1- 2t\langle \vec v, f(\zeta)\rangle +t^2\|\vec v\|^2\|f(\zeta)\|^2 } d\mu_{S^{n}}(\zeta)\\
&=  \lim_{t\to 0} \int_{S^{n}} \frac{1}{t} \cdot (f(\zeta) - 2 t\vec v + O(t^2))(1+ 2t\langle \vec v, f(\zeta)\rangle + O(t^2)) d\mu_{S^{n}}(\zeta)\\
&= -2 \vec v + 2 \int _{S^{n}}  \langle \vec v, f(\zeta)\rangle f(\zeta) d\mu_{S^{n}}(\zeta).
\end{align*}
Similarly, for the second equality, we have
\begin{align*}
F_{x} (\bm 0, \bm 0) (\vec v) &= \lim _{t\to 0} \frac{1}{t} \cdot F(t\vec v, \bm 0)\\
&=\lim _{t\to 0}  \int _{S^{n}} \frac{1}{t} \cdot  f (\zeta) (\frac{1-\|t\vec v\|^2}{\|\zeta - t\vec v\|^2})^{n} d\mu_{S^{n}}(\zeta)\\
&=\lim _{t\to 0} \int _{S^{n}} \frac{1}{t} \cdot f (\zeta) (1-t^2\|\vec v\|^2)^{n} (1-2t\langle \vec v, \zeta\rangle + t^2\|\vec v\|^2)^{-n} d\mu_{S^{n}}(\zeta)\\
&= \lim _{t\to 0} \int_{S^{n}} \frac{1}{t} \cdot f (\zeta) (1+O(t^2)) (1+2nt \langle \vec v, \zeta\rangle +O(t^2)) d\mu_{S^{n}}(\zeta)\\
&= 2n \int _{S^{n}} \langle \vec v, \zeta\rangle f(\zeta) d\mu_{S^{n}}(\zeta).
\end{align*}
\end{proof}




\section{Extremal length and width}\label{sec:ELW}
There is a wealth of sources containing background material on extremal length (see \cite{A73} or \cite[Appendix 4]{KL09}).
We will briefly summarize the necessary minimum.

\subsection*{Extremal length and extremal width}
Let $\Gamma$ be a family of (multi)curves\footnote{An element $\gamma \in \Gamma$ is a finite union of arcs or curves.} on $U \subseteq \C$.
Given a (measurable) conformal metric $\rho = \rho(z)|dz|$ on $U$, let
$$
L(\Gamma, \rho):= \inf_{\gamma\in \Gamma} L(\gamma, \rho),
$$
where $L(\gamma, \rho)$ stands for the $\rho$-length of $\gamma$.
The {\em extremal length} of $\Gamma$ is defined as
$$
\EL(\Gamma) := \sup_\rho \frac{L(\Gamma, \rho)^2}{A(U, \rho)},
$$
where $A(U, \rho)$ is the area of $U$ with respect to the measure $\rho^2$, and the supremum is taken over all $\rho$ subject to the condition $0 < A(U, \rho) < \infty$.

We call a conformal metric $\rho$ an {\em extremal metric} if $\EL(\Gamma) = \frac{L(\Gamma, \rho)^2}{A(U, \rho)}$.
The {\em extremal width} of $\Gamma$ is defined as the inverse of extremal length:
$$
\EW(\Gamma) = \frac{1}{\EL(\Gamma)}.
$$

\subsection*{Extremal distance and separating curves}
To avoid delicate discussions of the boundary behavior, we define an {\em admissible triple} $(U, E_1, E_2)$ by following conditions\footnote{See Appendix \ref{HD} for more general settings.}:
\begin{enumerate}
\item $U \subseteq \C$ is a bounded open set whose boundary consists of a finite number of Jordan curves;
\item Each of $E_1$ and $E_2$ is a union of boundary components of $U$ with $\partial U = E_1 \cup E_2$ and $E_1 \cap E_2 = \emptyset$.
\end{enumerate}

Let $\Gamma$ be the set of {\em connected arcs} in $U$ which join $E_1$ and $E_2$.
In other words, each $\gamma \in \Gamma$ shall have one end point in $E_1$ and one in $E_2$, and all other points shall be in $U$.
The {\em extremal distance} $\Mod_U(E_1, E_2)$ is defined by
$$
\Mod_U(E_1, E_2) = \EL(\Gamma).
$$

The extremal distance and the corresponding extremal metric can be computed explicitly using the solution for the Dirichlet problem as follows (see \cite[\S 4.9]{A73}).
Let $u$ be the unique harmonic function in $U$ equal to $1$ on $E_2$ and vanishing on $E_1$.
Then
\begin{itemize}
\item $\rho_0 = |\nabla u| = (u_x^2+u_y^2)^{1/2}$ is an extremal metric;
\item $\Mod_U(E_1, E_2) = \frac{1}{D(u)}$, where
$$
D(u) = \iint_U |\nabla u|^2 dxdy
$$
is the Dirichlet energy of the function $u$.
\end{itemize}

Let $\sigma$ be a finite union of closed curves in $U$.
We say $\sigma$ is {\em separating} in $(U, E_1, E_2)$ if there are two (relative) open sets $V_1, V_2$ of $\overline{U}$ with $E_1 \subseteq V_1$, $E_2 \subseteq V_2$ and $\overline{U} - \sigma = V_1 \cup V_2$.

Let $\Sigma^{sep}$ be the set of all separating (multi)curves.
It follows from the discussion in \cite[\S 4.9]{A73} that $\EL(\Sigma^{sep}) = D(u)$. Therefore, we have:
\begin{theorem}\label{thm:eds}
Let $(U, E_1, E_2)$ be an admissible triple. Then
$$
\Mod_U(E_1, E_2) = \EW(\Sigma^{sep}) = \frac{1}{\EL(\Sigma^{sep})}.
$$
\end{theorem}


\subsection*{Transformation rules}
Both extremal length and extremal width are conformal invariants.
They also behave nicely under holomorphic branched coverings by the following theorem (see \cite[Lemma 4.5]{KL09}).

\begin{theorem}\label{thm:tr}
Let $(U, E_1, E_2)$ and $(U', E_1', E_2')$ be two admissible triples.
Let $f: \overline{U} \longrightarrow \overline{U'}$ be a holomorphic branched covering of degree $d$ with $f(E_1) = E_1'$ and $f(E_2) = E_2'$. Then
$$
\Mod_{U'}(E_1', E_2') \leq d \cdot \Mod_{U}(E_1, E_2).
$$
\end{theorem}

\section{The uniform bound for barycentric extensions}\label{pf}
In this section, we will prove Theorem \ref{RationalLip}. 
Let $f: \hat\C \longrightarrow \hat\C$ be a rational map of degree $d$. Recall that 
$$
F(x, y) = \int _{S^{2}} M_{y} ^ {-1} ( f (M_{x}(\zeta))) d\mu_{S^{2}}(\zeta).
$$
The key is to establish the following estimate.
\begin{prop}\label{prop:rl}
Let $f: \hat\C\longrightarrow \hat\C$ be a rational map of degree $d$ with $\E f(\bm 0) = \bm 0$.
Then
$$
\| F_{y} (\bm 0, \bm 0)^{-1}\| \leq Md,
$$
for some universal constant $M$.
\end{prop}

Once the above estimate is established, Theorem \ref{RationalLip} follows immediately.
\begin{proof}[Proof of Theorem \ref{RationalLip} assuming Proposition \ref{prop:rl}]
If $\E f(\bm 0) = \bm 0$, by Proposition \ref{prop:cd}, 
$$
F_{x} (\bm 0, \bm 0) (\vec v) = 4 \int _{S^{2}} \langle \vec v, \zeta\rangle f(\zeta) d\mu_{S^{2}}(\zeta).
$$
Since $\zeta, f(\zeta) \in S^2$, $\|f(\zeta)\| =1$ and $\langle \vec v, \zeta\rangle \leq \|\vec v\|$.
Hence $\|F_{x} (\bm 0, \bm 0) \| \leq 4$.

By Proposition \ref{prop:rl}, $\| F_{y} (\bm 0, \bm 0)^{-1}\| \leq Md$.
Thus
$$
\|\Derivative\E f(\bm 0)\| \leq \|F_{y}(\bm 0, \bm 0)^{-1}\| \|F_{x}(\bm 0, \bm 0)\| \leq 4Md.
$$

Since $\Isom\Hyp^3 \times \Isom\Hyp^3$ acts transitively on pairs of points in $\Hyp^3$, by naturality of the barycentric extension, the theorem is proved.\footnote{From the proof of Proposition \ref{prop:rl}, we can take $M = \frac{27}{8\log 3}$. Thus, we can take the constant $C$ for Theorem \ref{RationalLip} as $\frac{27}{2\log 3}$.}
\end{proof}

To prove Proposition \ref{prop:rl}, we perform two further reductions.
First, by Proposition \ref{prop:cd}, 
$$
F_{y} (\bm 0, \bm 0) (\vec v) = -2 \vec v + 2 \int _{S^2} \langle \vec v, f(\zeta)\rangle f(\zeta) d\mu_{S^2}(\zeta).
$$
Let
$$
T: \vec v \mapsto \int _{S^2} \langle \vec v, f(\zeta)\rangle f(\zeta) d\mu_{S^2}(\zeta).
$$
Since $T$ is a self-adjoint operator with spectral radius $\leq 1$, to prove Proposition \ref{prop:rl}, it suffices to bound the largest eigenvalue from $1$.
After a change of variable, we assume the largest eigenvalue is associated to $\vec e_3 = \begin{bmatrix} 0 \\ 0 \\ 1 \end{bmatrix} \in \R^3$.

Second, we will relate the eigenvalue $\langle \vec e_3, T(\vec e_3)\rangle$ to the spherical area of certain sets.
Let
$$
U' = B(0, \sqrt{3}) - \overline{B(0, \frac{1}{\sqrt{3}})} \subseteq \C
$$
be an annulus.
Under the identification $\hat \C \cong S^2$ by the stereographic projection, $U'$ corresponds to the `belt'
$$
\{\begin{bmatrix} x \\ y \\ z \end{bmatrix} \in S^2: -\frac{1}{2} < z < \frac{1}{2}\} \subseteq S^2.
$$
Let $B_1' = \overline{B(0, \frac{1}{\sqrt{3}})}$ and $B_2' = \hat \C - B(0, \sqrt{3})$ be the two `caps'.
Let $U = f^{-1}(U')$, $B_1 = f^{-1}(B_1')$ and $B_2 = f^{-1}(B_2')$ be the corresponding preimages.

Let $V:= \mu_{S^2}(U)$, $V_1:= \mu_{S^2}(B_1)$ and $V_2:= \mu_{S^2}(B_2)$ be the normalized spherical area of the corresponding sets.
Note that 
$$
\langle \vec e_3, \langle \vec e_3, f(\zeta)\rangle f(\zeta)\rangle = (\langle \vec e_3, f(\zeta)\rangle)^2 \leq \begin{cases} \frac{1}{4} &\mbox{if } \zeta \in U \\ 
1 & \mbox{if } \zeta \in B_1 \cup B_2 \end{cases} .
$$
Therefore,
\begin{align*}
\langle \vec e_3, T(\vec e_3)\rangle &= \langle \vec e_3, \int _{S^2} \langle \vec e_3, f(\zeta)\rangle f(\zeta) d\mu_{S^2}(\zeta)\rangle\\
&=\int _{U} \langle \vec e_3,  \langle \vec e_3, f(\zeta)\rangle f(\zeta) \rangle d\mu_{S^2}(\zeta) + \int _{B_1 \cup B_2} \langle \vec e_3, \langle \vec e_3, f(\zeta)\rangle f(\zeta) \rangle d\mu_{S^2}(\zeta)\\
&\leq \frac{1}{4}\cdot V + 1\cdot (1-V) = 1-\frac{3}{4} V.
\end{align*}

We show we can bound $V$ from below.
\begin{lem}\label{lem:bV}
With the above notations,
$V \geq \frac{16 \log 3} {81 d}$.
\end{lem}

The proof uses the classical length-area principle developed by Gr\"otzsch and Ahlfors.
If $V$ is very small, then $V_1$ and $V_2$ are both big by the condition $\int _{S^2} f(\zeta) d\mu_{S^2}(\zeta) = \vec 0$.
Thus, any separating curve must bound a definite amount of spherical area. 
So the isoperimetric inequality gives a lower bound on its spherical length.

This gives a lower bound of $V$ in terms of the extremal width of separating curves.
Using Theorem \ref{thm:eds} and Theorem \ref{thm:tr}, we can bound the extremal width of separating curves from below. This gives the desired lower bound for $V$.

\begin{proof}[Proof of Lemma \ref{lem:bV}]
If $V \geq \frac{1}{3}$, then the lemma holds.
Thus, we may assume $V < \frac{1}{3}$.
Since $\E f(\bm 0) = \bm 0$, we have $\int _{S^{2}} f(\zeta) d\mu_{S^{2}}(\zeta) = \vec 0$. Therefore,
\begin{align*}
0 &= \langle \vec e_3, \int _{S^{2}} f(\zeta) d\mu_{S^{2}}(\zeta)\rangle\\
&=\int _{B_1} \langle \vec e_3,  f(\zeta) \rangle d\mu_{S^{2}}(\zeta) + \int _{U} \langle \vec e_3, f(\zeta) \rangle d\mu_{S^{2}}(\zeta) + \int _{B_2} \langle \vec e_3, f(\zeta) \rangle d\mu_{S^{2}}(\zeta)\\
& \leq (-\frac{1}{2}) \cdot V_1 + \frac{1}{2} \cdot V + 1 \cdot V_2 = (-\frac{1}{2})(1-V_2-V) + \frac{1}{2}V + V_2\\
&= \frac{3}{2}V_2+V-\frac{1}{2} \leq \frac{3}{2}V_2 - \frac{1}{6}.
\end{align*}
Hence $V_2 \geq \frac{1}{9}$.
Similarly, $V_1 \geq \frac{1}{9}$.

Let $E_1' = \partial B(0, \frac{1}{\sqrt{3}})$ and $E_2' = \partial B(0, \sqrt{3})$.
Let $E_1 = f^{-1}(E_1')$ and $E_2 = f^{-1}(E_2')$.
Then $(U, E_1, E_2)$ and $(U', E_1', E_2')$ are admissible triples.

\begin{figure}[ht]
 \centering
 \resizebox{0.7\linewidth}{!}{
    \def\svgwidth{\columnwidth}
    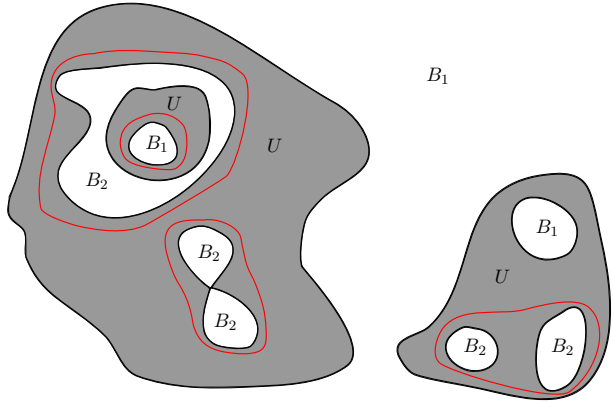

 }
 \caption{A schematic configuration of $U, B_1, B_2$. An example of a separating (multi)curve is colored in red.}
 \label{fig:ann}
\end{figure}

Let $\sigma \in \Sigma^{sep}$ be a separating curve for $(U, E_1, E_2)$, and let $\Omega$ be the set bounded by $\sigma$ with $B_1 \subseteq \Omega$ and $B_2\cap \Omega = \emptyset$.
Let $\rho_{S^2}$ be the standard spherical metric on $S^2$.
Since $V_1 = \mu_{S^2}(B_1), V_2= \mu_{S^2}(B_2) \geq \frac{1}{9}$, the spherical area of $\Omega$ is bounded between
$$
\frac{1}{9}\cdot 4\pi \leq A(\Omega, \rho_{S^2}) \leq \frac{8}{9}\cdot 4\pi.
$$
Since a spherical disk with area $\frac{4\pi}{9}$ has perimeter $\frac{8\sqrt{2}\pi}{9}$,
by the isoperimetric inequality for the sphere (see \cite[Troisi\`eme Partie, Chapitre I]{L22} or \cite{Sch39}),
$$
L(\sigma, \rho_{S^2}) \geq \frac{8\sqrt{2}\pi}{9}.
$$
Since this holds for any separating curve, we have
$$
\EL(\Sigma^{sep}) \geq \frac{L(\Sigma^{sep}, \rho_{S^2})^2}{A(U, \rho_{S^2})} \geq (\frac{8\sqrt{2}\pi}{9})^2/(4\pi\cdot V) = \frac{32\pi}{81V}.
$$
By Theorem \ref{thm:eds}, we have
\begin{align}\label{eqn:ineq1}
\Mod_U(E_1, E_2) = \frac{1}{\EL(\Sigma^{sep})} \leq \frac{81V}{32\pi}.
\end{align}

Since $U'$ is an annulus, the extremal distance $\Mod_{U'}(E_1', E_2') = \frac{\log 3}{2\pi}$ equals to the modulus of $U'$.
Since $f: \overline{U} \longrightarrow \overline{U'}$ is a degree $d$ branched covering with $f(E_1) = E_1'$ and $f(E_2) = E_2'$, by Theorem \ref{thm:tr},
\begin{align}\label{eqn:ineq2}
\Mod_{U}(E_1, E_2) \geq \frac{1}{d} \cdot \Mod_{U'}(E_1', E_2') = \frac{\log 3}{2\pi d}. 
\end{align}

Combining the inequalities \ref{eqn:ineq1} and \ref{eqn:ineq2}, we get
$$
V \geq \frac{16\log 3}{81d}.
$$
\end{proof}

We are now ready to prove Proposition \ref{prop:rl}.
\begin{proof}[Proof of Proposition \ref{prop:rl}]
Since $T$ is self-adjoint, all of its eigenvalues are real.
By Lemma \ref{lem:bV},
the largest eigenvalue of $T$ is at most $1-\frac{3}{4}V \leq 1-\frac{4\log 3}{27d}$.

Since $F_{y} (\bm 0, \bm 0) = 2T - 2I$, the smallest absolute value of its eigenvalues is at least $\frac{8\log 3}{27d}$.
Thus,
$\|F_{y} (\bm 0, \bm 0)^{-1}\| \leq \frac{27}{8\log 3} \cdot d$.
\end{proof}

\section{Limits of barycentric extensions on $\Hyp^3$}\label{CC}
In this section, we will study the limit of barycentric extensions $\E f_n$ on $\Hyp^3$ for a sequence of rational maps $f_n\in \Rat_d(\C)$.
We will give a condition for $\E f_n$ to converge compactly to a map $F: \Hyp^3 \longrightarrow \Hyp^3$ (see Theorem \ref{CompactConvergence}).

\subsection*{The space of rational maps}
The space $\Rat_d(\C)$ of rational maps of degree $d\geq 1$ is an open variety of $\Proj^{2d+1}_\C$.
More concretely, fixing a coordinate system on $\hat\C \cong \Proj^1_\C$, then a rational map can be expressed as a ratio of homogeneous polynomials 
$$
f(z:w) = (P(z,w): Q(z,w)),
$$ 
where $P$ and $Q$ have degree $d$ with no common divisors. Using the coefficients of $P$ and $Q$ as parameters, then
$$
\Rat_d(\C) = \Proj^{2d+1}_\C - V(\Res),
$$
where $V(\Res)$ is the vanishing locus for the resultant of $P$ and $Q$.

One natural compactification of $\Rat_d(\C)$ is $\overline{\Rat_d(\C)} = \Proj^{2d+1}_\C$. 
Every map $f\in \overline{\Rat_d(\C)}$ determines the coefficients of a pair of homogeneous polynomials. We write
$$
f= (P: Q) = (Hp:Hq) = H\varphi_f,
$$
where $H = \gcd (P, Q)$ and $\varphi_f = (p:q)$ is a rational map of degree at most $d$.
A zero of $H$ is called a {\em hole} of $f$ and the set of zeros of $H$ is denoted by
$\mathcal{H}(f)$.
We define the degree of $f\in \overline{\Rat_d(\C)}$ as the deg of $\varphi_f$.


The following lemma is well-known (see \cite[Lemma 4.2]{DeM05}):
\begin{lem}\label{comcon}
Let $f_n\in\Rat_d(\C)$ converge to $f \in \overline{\Rat_d(\C)}$. Then $f_n$ converges compactly to $\varphi_f$ on $\hat\C - \mathcal{H}(f)$.
\end{lem}

The next lemma follows from the proof of \cite[Lemma 4.5 and 4.6]{DeM05}:
\begin{lem}\label{lem:ns}
Let $f_n\in\Rat_d(\C)$ converge to $f \in \overline{\Rat_d(\C)}$ and $a\in \mathcal{H}(f)$.
Let $U$ be a neighborhood of $a$.
\begin{itemize}
\item If $\deg(\varphi_f) \geq 1$, then there exists $K>1$ so that $f_n(U) = \hat\C$ for all $n\geq K$. 
Moreover, there exists a sequence of critical points $c_n$ of $f_n$  converging to $a$.
\item If $\varphi_f\equiv C$ is a constant map, then for any compact set $M\subseteq \hat\C - \{C\}$, there exists $K>1$ so that $M \subseteq f_n(U)$ for all $n \geq K$.
\end{itemize}
\end{lem}

We also need the following lemma (see \cite[Lemma 2.6]{DeM07}).
\begin{lem}\label{lem:avoh}
Let $f_n\in\Rat_{d_1}(\C)$ and $g_n \in \Rat_{d_2}(\C)$ with $f_n \to f\in \overline{\Rat_{d_1}(\C)}$ and $g_n \to g \in \overline{\Rat_{d_2}(\C)}$. 
If $\varphi_f$ is not a constant map taking value in $\mathcal{H}(g)$, then $g_n \circ f_n$ converges compactly to $\varphi_g \circ \varphi_f$ on $\hat\C - \mathcal{H}(f) - \varphi_f^{-1}(\mathcal{H}(g))$.
\end{lem}

Since $\mathcal{H}(f)$ and $\varphi_f^{-1}(\mathcal{H}(g))$ are both finite sets if $\varphi_f$ is not a constant function taking value in $\mathcal{H}(g)$, the above lemma immediately implies that
\begin{lem}\label{WeakLimit}
Let $f_n\in\Rat_{d_1}(\C)$ and $g_n \in \Rat_{d_2}(\C)$ with $f_n \to f\in \overline{\Rat_{d_1}(\C)}$ and $g_n \to g \in \overline{\Rat_{d_2}(\C)}$.
If $\varphi_f$ is not a constant map taking value in $\mathcal{H}(g)$, then $(g_n \circ f_n)_*\mu_{S^2}$ converges weakly to $(\varphi_g \circ \varphi_f)_*\mu_{S^2}$.
\end{lem}

\subsection*{The limit of $\E f_n$ on $\Hyp^3$}
Using Lemma \ref{WeakLimit}, we have:
\begin{prop}\label{bdd}
Let $f_n \in \Rat_d(\C)$. 
Then $\E f_n (\bm 0)$ is uniformly bounded 
if and only if degree $0$ maps are not in the limit set of $\{f_n\}$ in $\overline{\Rat_d(\C)}$.
\end{prop}
\begin{proof}
If a degree $0$ map is in the limit set, then after passing to a subsequence, we assume $f_n \to f\in \overline{\Rat_d(\C)}$ of degree $0$.

Suppose for contradiction that $\E f_n (\bm 0)$ stays bounded.
Let $M_n \in \PSL_2(\C)$ with $\E f_n (\bm 0) = M_n(\bm 0)$.
Then after passing to a further subsequence, $M_n\to M\in \PSL_2(\C)$.
By naturality of the barycentric extension,
$$
\E (M_n^{-1} \circ f_n) (\bm 0) = M_n^{-1} \circ \E f_n (\bm 0) = \bm 0.
$$
Hence, the measure $(M_n^{-1} \circ f_n)_* \mu_{S^2}$ is balanced for all $n$.

By Lemma \ref{WeakLimit}, $(M_n^{-1} \circ f_n)_* \mu_{S^2}$ converges weakly to the delta measure $(M^{-1}\circ \varphi_f)_*\mu_{S^2}$, giving a contradiction.

Conversely, assume that $\E f_n (\bm 0)$ is unbounded.
Let $M_n \in \PSL_2(\C)$ with $\E f_n (\bm 0) = M_n(\bm 0)$.
Then after passing to a subsequence, $M_n\to M \in \overline{\Rat_1(\C)}$ of degree $0$.
By naturality of the barycentric extension,
$$
\E (M_n^{-1} \circ f_n) (\bm 0) = M_n^{-1} \circ \E f_n (\bm 0) = \bm 0.
$$
Hence, the measure $(M_n^{-1} \circ f_n)_* \mu_{S^2}$ is balanced for all $n$.

Suppose for contradiction that degree $0$ maps are not in the limit set of $\{f_n\}$. 
Then after passing to a further subsequence, $f_n \to f\in \overline{\Rat_d(\C)}$ of degree $\geq 1$.
By Lemma \ref{WeakLimit},
$(M_{n}^{-1} \circ f_{n})_* \mu_{S^2}$ converges to the delta measure $(\varphi_{M^{-1}}\circ \varphi_f)_*\mu_{S^2}$, giving a contradiction.
\end{proof}

Using Theorem \ref{RationalLip}, we show:
\begin{theorem}\label{CompactConvergence}
Let $f_n \in \Rat_d(\C)$ converge to $f \in \overline{\Rat_d(\C)}$ with $\deg(\varphi_f) \geq 1$. Then $\E f_n$ converges compactly to $\E \varphi_f:\Hyp^3 \longrightarrow \Hyp^3$.
\end{theorem}
\begin{proof}
We first claim that $\E f_n$ converges to $\E \varphi_f$ pointwise.
By naturality of the barycentric extension, it suffices to show that $\E f_n(\bm 0)$ converges to $\bm 0$ under the assumption that $\E\varphi_f(\bm 0) = \bm 0$.

Let $M_n\in \PSL_2(\C)$ with $M_n(\bm 0) = \E f_n(\bm 0)$.
By Proposition \ref{bdd}, $M_n$ is bounded in $\PSL_2(\C)$.
Suppose for contradiction that $M_n(\bm 0)$ does not converge to $\bm 0$.
Then after passing to a subsequence, we assume that $M_n \to M \in \PSL_2(\C)$ with $M(\bm 0) \neq \bm 0$.

By Lemma \ref{WeakLimit}, $(M_{n}^{-1}\circ f_{n})_*\mu_{S^2}$ converges weakly to $(M^{-1}\circ \varphi_f)_*\mu_{S^2}$. 
Since the measure $(M_{n}^{-1}\circ f_{n})_*\mu_{S^2}$ is balanced for any $n$, $(M^{-1}\circ \varphi_f)_*\mu_{S^2}$ is also balanced. 
Hence by naturality, 
$$
\E \varphi_f (\bm 0) = M(\bm 0) \neq \bm 0,
$$ 
which is a contradiction.
Hence $\E f_n$ converges to $\E \varphi_f$ pointwise.

By Theorem \ref{RationalLip}, the sequence $\E f_n$ is $Cd$-Lipschitz. Hence, the pointwise convergence can be promoted to uniform convergence on any compact set by Arzel\`a-Ascoli theorem.
Therefore, $\E f_n$ converges compactly to $\E \varphi_f$.
\end{proof}

As an application of Theorem \ref{CompactConvergence}, we record the following proposition for future references.
\begin{prop}\label{lem:pi}
Let $f_n \in \Rat_d(\C)$ converge to $f \in \overline{\Rat_d(\C)}$ with $\deg(\varphi_f) \geq 1$. Then there exists $R>0$ and $x_n \in \E f_n^{-1}(\bm 0)$ so that for all sufficiently large $n$,
$$
\dist_{\Hyp^3} (\bm 0, x_n) \leq R.
$$
\end{prop}
\begin{proof}
Fix a Euclidean ball $B(\bm 0, r) \subseteq \R^3$ with $r<1$ such that 
$$
\E \varphi_f ^{-1}(\bm 0) \cap (B(\bm 0, 1) - B(\bm 0, r)) = \emptyset.
$$ 
Such $r$ exists as $\E \varphi_f$ is proper.

It suffices to show that for all sufficiently large $n$, $\E f_n^{-1}(\bm 0) \cap \overline{B(\bm 0, r)} \neq \emptyset$.
Suppose not. Then after passing to a subsequence, we can define
\begin{align*}
F_n: \overline{B(\bm 0, r)} &\longrightarrow S^2\\
x&\mapsto \E f_n (x) / \|\E f_n (x) \|,
\end{align*}
and
\begin{align*}
F: \partial B(\bm 0, r) &\longrightarrow S^2 \\
x &\mapsto \E \varphi_f (x) / \|\E \varphi_f (x)\|.
\end{align*}

Since $\E f_n$ converges uniformly to $\E \varphi_f$ on $\partial B(\bm 0, r)$ by Theorem \ref{CompactConvergence}, $F_n |_{\partial B(\bm 0, r)}$ is homotopic to $F$. 
Since there are no preimages of $\bm 0$ of $\E \varphi_f$ in $B(\bm 0, 1) - B(\bm 0, r)$, $F$ is homotopic to $\varphi_f$ on $\hat \C$.
So $F_n |_{\partial B(\bm 0, r)}$ is homotopic to $\varphi_f$.

Since $\varphi_f$ has homological degree $\geq 1$, $F_n |_{\partial B(\bm 0, r)}$ has homological  degree $\geq 1$ as well. 
Therefore, $F_n$ cannot be extended to a continuous map from $\overline{B(\bm 0, r)}$ to $S^2$, giving a contradiction. This proves the proposition.
\end{proof}


\section{The Carath\'eodory convergence of annuli}\label{sec:CA}
In this section, we will first define the Carath\'eodory topology for annuli  (cf. Carath\'eodory topology for pointed disks in \cite[\S 5]{McM94}).
We will then use it to study the barycentric extensions on large scale.

\subsection*{Carath\'eodory topology on the spaces of annuli}
Let $E \subseteq \C$ be a {\em continuum}, i.e., a connected compact set.
Let $V \subseteq \C - E$ be the unbounded component.
We define the {\em hull} of $E$ by $\Hull(E) = \C - V$.

Note that the $\Hull(E)$ is a {\em full continuum}, i.e., a connected compact set whose complement is also connected.

\begin{defn}
Let $A_n = U_n - K_n$ be an annulus where $U_n$ is a topological disk of $\C$, and $K_n \subseteq U_n$ is a full continuum.
We say $A_n$ converges to an annulus $A = U-K$ (in the Carath\'eodory topology) if
\begin{enumerate}
\item $K_n$ converges in the Hausdorff topology on compact subset of $\hat\C$ to $E$, and $K = \Hull(E)$; and
\item $\hat\C - U_n$ converges in the Hausdorff topology on compact subset of $\hat\C$ to $D$, and $U$ is the component of $\hat\C - D$ containing $K$.
\end{enumerate}
\end{defn}

Equivalently, $A_n = U_n - K_n$ converges to $A = U-K$ if 
there exist $u_n \in K_n$ and $u\in K$ so that
\begin{itemize}
\item $(U_n, u_n)$ converges to $(U,u)$ in the Carath\'eodory topology; and
\item $K_n$ converges to $E$ in the Hausdorff topology, and $K = \Hull(E)$.
\end{itemize}

\begin{figure}[ht]
 \centering
 \resizebox{0.6\linewidth}{!}{
    \def\svgwidth{\columnwidth}
\begingroup%
  \makeatletter%
  \providecommand\color[2][]{%
    \errmessage{(Inkscape) Color is used for the text in Inkscape, but the package 'color.sty' is not loaded}%
    \renewcommand\color[2][]{}%
  }%
  \providecommand\transparent[1]{%
    \errmessage{(Inkscape) Transparency is used (non-zero) for the text in Inkscape, but the package 'transparent.sty' is not loaded}%
    \renewcommand\transparent[1]{}%
  }%
  \providecommand\rotatebox[2]{#2}%
  \newcommand*\fsize{\dimexpr\f@size pt\relax}%
  \newcommand*\lineheight[1]{\fontsize{\fsize}{#1\fsize}\selectfont}%
  \ifx\svgwidth\undefined%
    \setlength{\unitlength}{841.88976378bp}%
    \ifx\svgscale\undefined%
      \relax%
    \else%
      \setlength{\unitlength}{\unitlength * \real{\svgscale}}%
    \fi%
  \else%
    \setlength{\unitlength}{\svgwidth}%
  \fi%
  \global\let\svgwidth\undefined%
  \global\let\svgscale\undefined%
  \makeatother%
  \begin{picture}(1,0.70707071)%
    \lineheight{1}%
    \setlength\tabcolsep{0pt}%
    \put(0,0){\includegraphics[width=\unitlength,page=1]{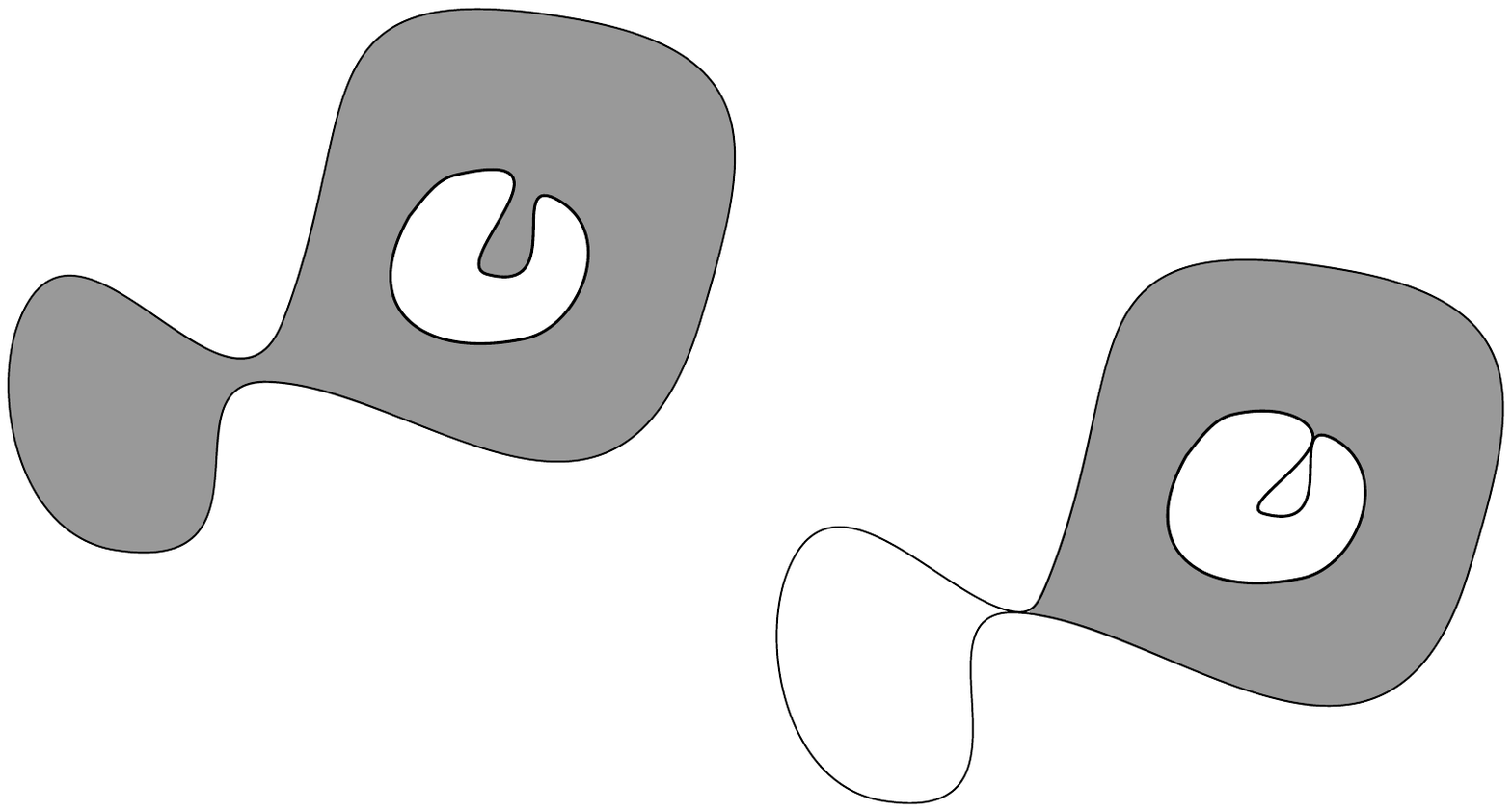}}%
    \put(0.32070707,0.55944363){\color[rgb]{0,0,0}\makebox(0,0)[lt]{\lineheight{1.25}\smash{\begin{tabular}[t]{l}$A_n$\end{tabular}}}}%
    \put(0.80304032,0.40418069){\color[rgb]{0,0,0}\makebox(0,0)[lt]{\lineheight{1.25}\smash{\begin{tabular}[t]{l}$A$\end{tabular}}}}%
  \end{picture}%
\endgroup%

 }
 \caption{An illustration of the convergence of annuli.}
 \label{fig:CA}
\end{figure}

Let $\Mod(A)$ be the modulus of an annulus $A$.
The following theorem follows from the proof of \cite[Theorem 5.8]{McM94}.
We include the proof here for completeness.
\begin{theorem}\label{MB}
Given $m>0$, the space of annuli $A \subseteq \C$ with $\Mod(A) \geq m$ is compact up to affine transformation.

More precisely, let $A_n = U_n - K_n$ with $\Mod(A_n) \geq m$.
Then after passing to a subsequence, either
\begin{enumerate}
\item $K_n$ is degenerate for all $n$, and $A_n$ normalized so that $K_n = \{0\}$ and $B(0, 1) \subseteq U_n$, has a convergent subsequence; or
\item $K_n$ is nondegenerate for all $n$, and $A_n$ normalized so that $0\in K_n$ and the Euclidean diameter of $K_n$ is equal to $1$, has a convergent subsequence.
\end{enumerate}
\end{theorem}
\begin{proof}
For the first case, after passing to a subsequence, $(U_n, 0)$ converges to some pointed disk $(U, 0)$ in the Carath\'eodory topology by \cite[Theorem 5.2]{McM94}.
Thus $A_n$ converges to $A = U - \{0\}$, and $\Mod(A) = \infty \geq m$.

For the second case, since the Euclidean diameter of $0\in K_n$ is equal to $1$ and $\Mod(A_n) \geq m$, the Euclidean distance from $0$ to $\partial U_n$ is greater than $C>0$, where $C$ is a constant depending only on $m$ by \cite[Theorem 2.5]{McM94}.
Thus, after passing to a subsequence, $(U_n, 0)$ converges to some pointed disk $(U, 0)$ in the Carath\'eodory topology by \cite[Theorem 5.2]{McM94}.

The lower bound $\Mod(A_n) \geq m$ provides an upper bound on the diameter of $K_n$ in the hyperbolic metric on $U_n$, by \cite[Theorem 2.4]{McM94}.
Thus, after passing to a subsequence, the Hausdorff limit $E$ of $K_n$ is a compact set of $U$ by \cite[Theorem 5.3]{McM94}, so $K = \Hull(E)$ is also a compact set of $U$.
Therefore $A_n$ converges to $A = U - K$.

Let $h_n: A(m) \longrightarrow A_n$ be a univalent map from the standard round annulus $A(m)$ of modulus $m$ into $A_n$. Then one can extract a limiting injection into $A$ by \cite[Corollary 2.8]{McM94}, so $\Mod(A) \geq m$.
\end{proof}

\subsection*{Carath\'eodory topology on functions on annuli}
The following notion is the analogue of the Carath\'eodory topology for functions on pointed disks \cite[\S 5]{McM94}.
\begin{defn}
Let $f_n: A_n \longrightarrow \C$ be a sequence of holomorphic functions on annuli $A_n \subseteq \C$.
Then $f_n$ is said to converge to $f: A \longrightarrow \C$ if
\begin{itemize}
\item $A_n$ converges to $A$; and
\item for all sufficiently large $n$, $f_n$ converges to $f$ uniformly on compact subsets of $A$.
\end{itemize}
\end{defn}

Note that any compact set $M \subseteq A$ is eventually contained in $A_n$, so $f_n$ is defined on $M$ for all sufficiently large $n$.

\begin{lem}\label{lem:cdec}
Let $f_n: A_n \longrightarrow A' \subseteq \C$ be a sequence of degree $e$ holomorphic coverings between annuli $A_n$ and $A'$.
If $f_n$ converges to $f: A \longrightarrow A'$, then $f: A \longrightarrow A'$ is a degree $e$ covering.
\end{lem}
\begin{proof}
We will first show that $f: A \longrightarrow A'$ is proper.
Post-composing with a uniformization map, we may assume $A' = B(0, R) - \overline{B(0, 1)}$.

Let $\epsilon > 0$ and $A'_\epsilon := B(0, R-\epsilon) - \overline{B(0, 1+\epsilon)}$.
Since $f_n$ is a degree $e$ covering map, $A_n - f_n^{-1}(\overline{A'_\epsilon})$ consists of two annuli $A_{n, \epsilon}^\pm$, with $\Mod(A_{n, \epsilon}^\pm)$ uniformly bounded from below.
Therefore, $f_n^{-1}(\overline{A'_\epsilon})$ converges to a compact subset of $A$ by the same argument in Theorem \ref{MB}.
Since $f_n$ converges compactly to $f$, $f^{-1}(\overline{A'_\epsilon})$ is compact in $A$.
Since any compact set of $A'$ is contained in $A'_\epsilon$ for some $\epsilon$, $f$ is proper.

By the argument principle, the degree is $e$. Thus $f: A \longrightarrow A'$ is a degree $e$ covering.
\end{proof}

The next lemma is useful to bound the degrees for limits of rational maps.
\begin{lem}\label{DB}
Let $f_n \in \Rat_d(\C)$ with $f_n \to f\in \overline{\Rat_d(\C)}$.
Let $A_n, A, A' \subseteq \C$ be annuli with $A_n$ converging to $A$.
If $f_n: A_n \longrightarrow A'$ is a degree $e$ covering, then $\varphi_f : A \longrightarrow A'$ is a degree $e$ covering.
In particular, $\deg(\varphi_f) \geq e$.
\end{lem}
\begin{proof}
We claim $A \cap \mathcal{H}(f) = \emptyset$.
Suppose not, and let $U \subseteq  A$ be a neighborhood of a hole $a \in \mathcal{H}(f)$.
Let $M = \hat\C$ if $\varphi_f$ is non-constant, and let $M$ be a compact subset of $\hat \C - \{C\}$ with $M \cap (\hat \C - A') \neq \emptyset$ if $\varphi_f \equiv C$.
By Lemma \ref{lem:ns}, $M \subseteq f_n(U)$ for sufficiently large $n$. This is a contradiction as $f_n(U) \subseteq A'$.

Therefore, by Lemma \ref{comcon}, $f_n:A_n \longrightarrow A'$ converges to $\varphi_f:A \longrightarrow A'$.
By Lemma \ref{lem:cdec}, $\varphi_f$ is a degree $e$ covering.
\end{proof}

\subsection*{Maps on annuli with large modulus}
In the following, we will discuss how the Carath\'eodory convergence allows us to control barycentric extensions on large scale.

\label{defA}
Given two points $x, y\in \Hyp^3$, we can associate an annulus, denoted by $A(x,y)$, as follows. Let $H_x$ and $H_y$ be the two geodesic planes perpendicular to the geodesic segment $[x,y]$ that pass through $x$ and $y$ respectively. 
The boundaries of $H_x$ and $H_y$ in $\hat\C \cong \partial \Hyp^3$ are two circles $C_x$ and $C_y$.
The annulus $A(x,y)$ is defined as the region bounded by $C_x$ and $C_y$.
Note that $\Mod(A(x,y)) = \dist_{\Hyp^3}(x,y) / 2\pi$.


We say a pair of points $(x,y)$ in $\Hyp^3$ {\em approximates} an annulus $A$ if
\begin{itemize}
\item $A(x, y)$ is essentially contained in $A$; and
\item $\Mod(A) \leq \Mod(A(x,y)) + \frac{5\log 2}{2\pi}$.
\end{itemize}
Here {\em essential} means $\pi_1(A(x, y))$ injects into $\pi_1(A)$.

By \cite[Theorem 2.1]{McM94}, any annulus $A$ with sufficiently large modulus can be approximated by a pair of points in $\Hyp^3$ (see also \cite[\S 4.11-12]{A73}).

The following proposition shows how the maps on large annuli control the corresponding  barycentric extensions.
We will use it to study the limiting map on rescalings of $\Hyp^3$ in \S \ref{gl}.
\begin{prop}\label{lem:annb}
Let $f_n \in \Rat_d(\C)$. Let $x_n, y_n \in \Hyp^3$ with $\dist_{\Hyp^3}(x_n, y_n) \to \infty$.
Assume $f_n : A_n \longrightarrow A_n':=A(x_n, y_n)$ is a degree $e$ covering, and $(a_n, b_n)$ approximates $A_n$.
Then $\dist_{\Hyp^3}(x_n, \E f_n(a_n))$ and $\dist_{\Hyp^3}(y_n, \E f_n(b_n))$ are bounded.
\end{prop}
\begin{proof}
We shall prove $\dist_{\Hyp^3}(x_n, \E f_n(a_n))$ is bounded, the second statement is proved similarly.

By naturality, we assume $x_n = a_n = \bm 0$ with $A_n' = B(0, R'_n) - \overline{B(0,1)}$ and $A(a_n, b_n) = B(0, R_n) - \overline{B(0,1)} \subseteq A_n \subseteq \C$.
Let $D_n$ be the bounded component of $\C - A_n$. We assume that $\partial D_n$ is mapped to $\partial B(0,1)$.

Under the above normalization, by Proposition \ref{bdd}, it suffices to show that any limit point of $f_n$ in $\overline{\Rat_d(\C)}$ has degree $\geq 1$.

Suppose for contradiction that this is not the case. Then after passing to a subsequence, we assume $f_n \to f\in \overline{\Rat_d(\C)}$ of degree $0$.

\begin{figure}[ht]
 \centering
 \resizebox{0.8\linewidth}{!}{
    \def\svgwidth{\columnwidth}
\begingroup%
  \makeatletter%
  \providecommand\color[2][]{%
    \errmessage{(Inkscape) Color is used for the text in Inkscape, but the package 'color.sty' is not loaded}%
    \renewcommand\color[2][]{}%
  }%
  \providecommand\transparent[1]{%
    \errmessage{(Inkscape) Transparency is used (non-zero) for the text in Inkscape, but the package 'transparent.sty' is not loaded}%
    \renewcommand\transparent[1]{}%
  }%
  \providecommand\rotatebox[2]{#2}%
  \newcommand*\fsize{\dimexpr\f@size pt\relax}%
  \newcommand*\lineheight[1]{\fontsize{\fsize}{#1\fsize}\selectfont}%
  \ifx\svgwidth\undefined%
    \setlength{\unitlength}{841.88976378bp}%
    \ifx\svgscale\undefined%
      \relax%
    \else%
      \setlength{\unitlength}{\unitlength * \real{\svgscale}}%
    \fi%
  \else%
    \setlength{\unitlength}{\svgwidth}%
  \fi%
  \global\let\svgwidth\undefined%
  \global\let\svgscale\undefined%
  \makeatother%
  \begin{picture}(1,0.70707071)%
    \lineheight{1}%
    \setlength\tabcolsep{0pt}%
    \put(0,0){\includegraphics[width=\unitlength,page=1]{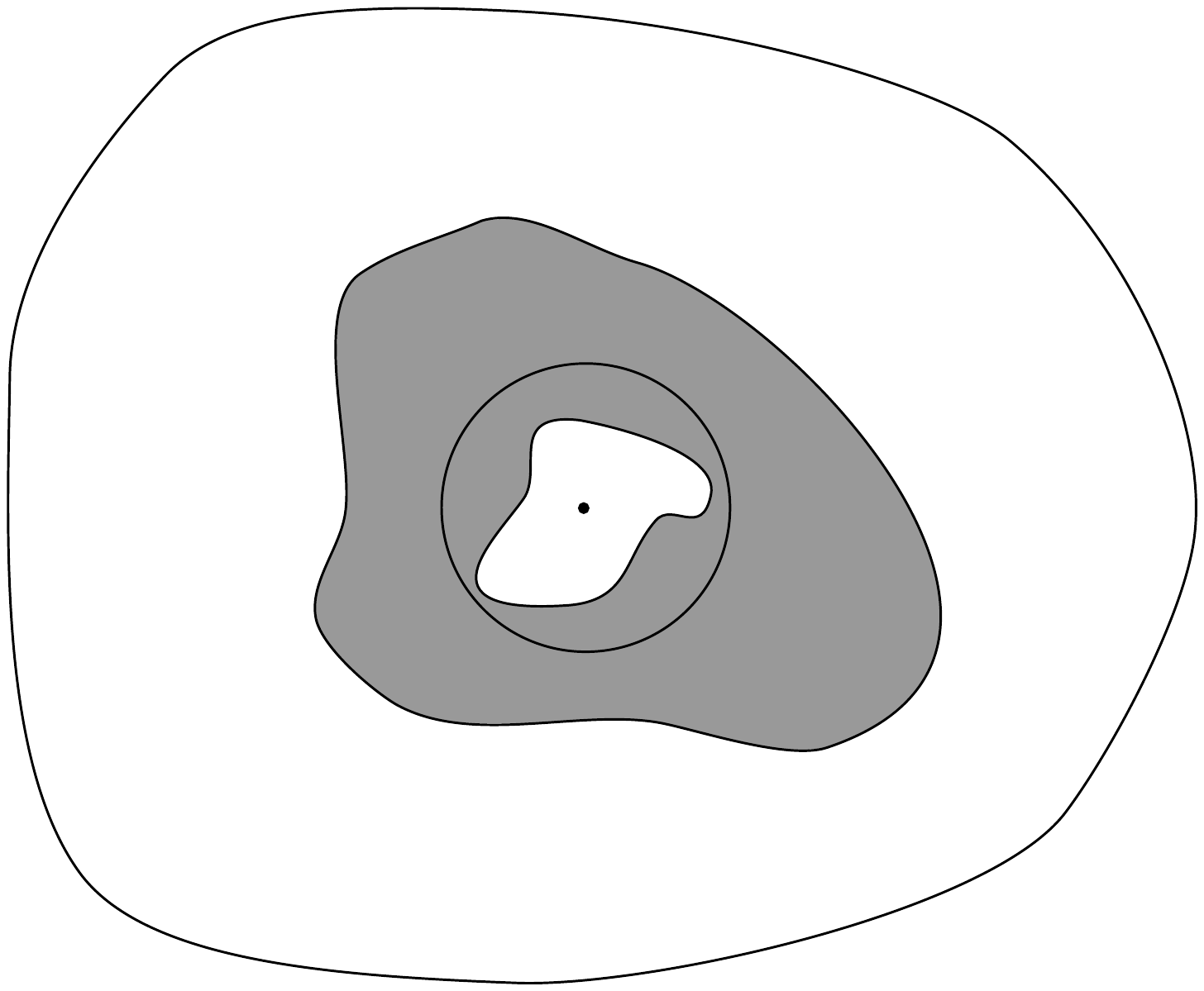}}%
    \put(0.3502476,0.47242611){\color[rgb]{0,0,0}\makebox(0,0)[lt]{\lineheight{1.25}\smash{\begin{tabular}[t]{l}$\widetilde A_n$\end{tabular}}}}%
    \put(0.44115522,0.30956276){\color[rgb]{0,0,0}\makebox(0,0)[lt]{\lineheight{1.25}\smash{\begin{tabular}[t]{l}$D_n$\end{tabular}}}}%
    \put(0,0){\includegraphics[width=\unitlength,page=2]{Ain.pdf}}%
    \put(0.4996437,0.22359818){\color[rgb]{0,0,0}\makebox(0,0)[lt]{\lineheight{1.25}\smash{\begin{tabular}[t]{l}$C_{a_n} = \partial B(0,1)$\end{tabular}}}}%
    \put(0.76454481,0.48636825){\color[rgb]{0,0,0}\makebox(0,0)[lt]{\lineheight{1.25}\smash{\begin{tabular}[t]{l}$C_{b_n} = \partial B(0,R_n)$\end{tabular}}}}%
    \put(0.75951337,0.61708097){\color[rgb]{0,0,0}\makebox(0,0)[lt]{\lineheight{1.25}\smash{\begin{tabular}[t]{l}$\partial A_n \cap f_n^{-1}(C_{y_n})$\end{tabular}}}}%
    \put(0.50816242,0.28749092){\color[rgb]{0,0,0}\makebox(0,0)[lt]{\lineheight{1.25}\smash{\begin{tabular}[t]{l}$\partial A_n \cap f_n^{-1}(C_{x_n})$\end{tabular}}}}%
  \end{picture}%
\endgroup%

    \def\svgwidth{\columnwidth}
\begingroup%
  \makeatletter%
  \providecommand\color[2][]{%
    \errmessage{(Inkscape) Color is used for the text in Inkscape, but the package 'color.sty' is not loaded}%
    \renewcommand\color[2][]{}%
  }%
  \providecommand\transparent[1]{%
    \errmessage{(Inkscape) Transparency is used (non-zero) for the text in Inkscape, but the package 'transparent.sty' is not loaded}%
    \renewcommand\transparent[1]{}%
  }%
  \providecommand\rotatebox[2]{#2}%
  \newcommand*\fsize{\dimexpr\f@size pt\relax}%
  \newcommand*\lineheight[1]{\fontsize{\fsize}{#1\fsize}\selectfont}%
  \ifx\svgwidth\undefined%
    \setlength{\unitlength}{841.88976378bp}%
    \ifx\svgscale\undefined%
      \relax%
    \else%
      \setlength{\unitlength}{\unitlength * \real{\svgscale}}%
    \fi%
  \else%
    \setlength{\unitlength}{\svgwidth}%
  \fi%
  \global\let\svgwidth\undefined%
  \global\let\svgscale\undefined%
  \makeatother%
  \begin{picture}(1,0.70707071)%
    \lineheight{1}%
    \setlength\tabcolsep{0pt}%
    \put(0,0){\includegraphics[width=\unitlength,page=1]{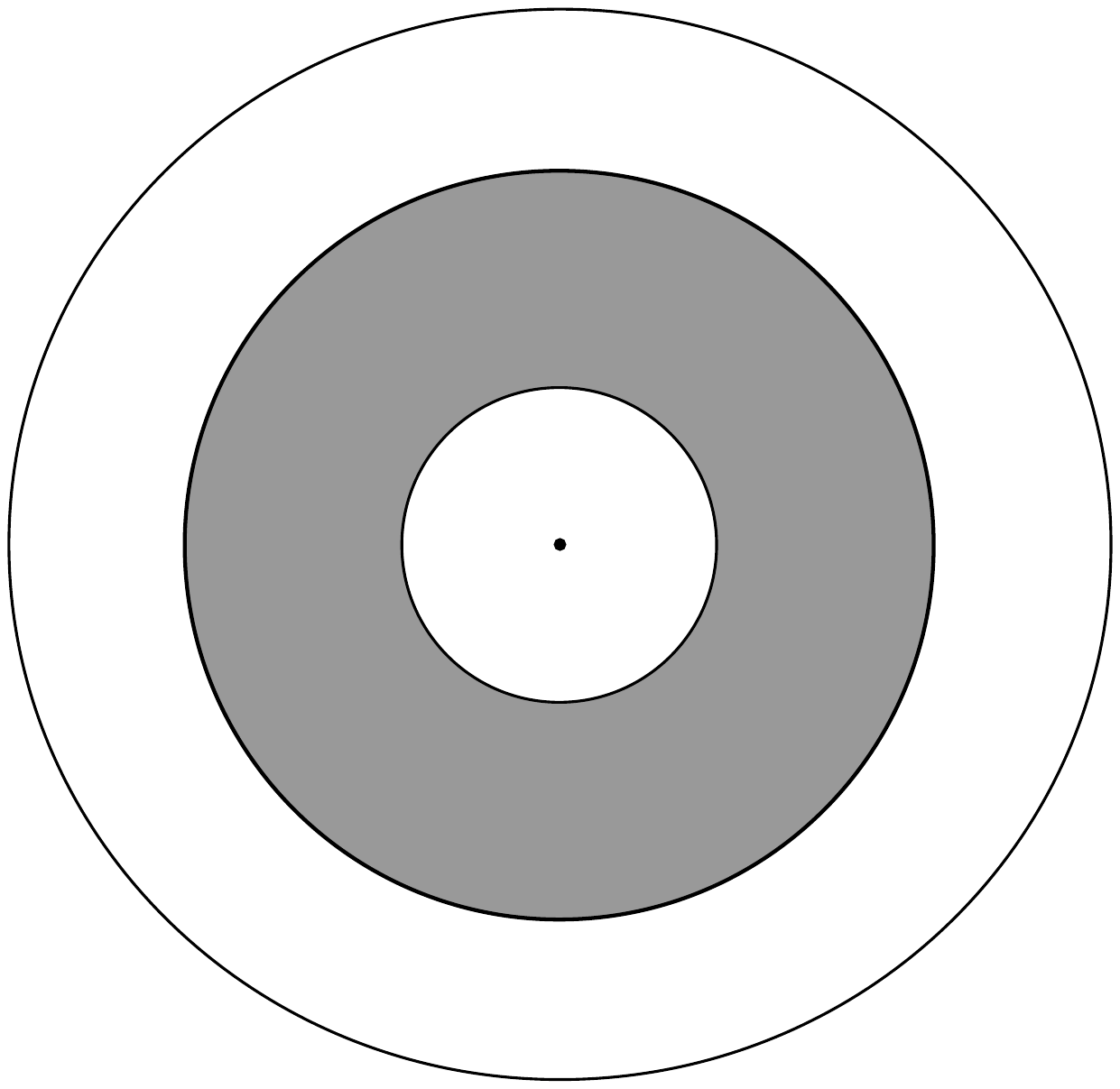}}%
    \put(0.47604581,0.22742537){\color[rgb]{0,0,0}\makebox(0,0)[lt]{\lineheight{1.25}\smash{\begin{tabular}[t]{l}$C_{x_n} = \partial B(0,1)$\end{tabular}}}}%
    \put(0.71381925,0.65824866){\color[rgb]{0,0,0}\makebox(0,0)[lt]{\lineheight{1.25}\smash{\begin{tabular}[t]{l}$C_{y_n} = \partial B(0,R'_n)$\end{tabular}}}}%
    \put(0.35995894,0.47039772){\color[rgb]{0,0,0}\makebox(0,0)[lt]{\lineheight{1.25}\smash{\begin{tabular}[t]{l}$\widetilde A'$\end{tabular}}}}%
    \put(0.66052472,0.54328938){\color[rgb]{0,0,0}\makebox(0,0)[lt]{\lineheight{1.25}\smash{\begin{tabular}[t]{l}$\partial B(0, K)$\end{tabular}}}}%
  \end{picture}%
\endgroup%

 }
 \caption{}
 \label{fig:ain}
\end{figure}

Let $\widetilde A':= B(0, K) - \overline{B(0,1)}$ and $\widetilde A_n := f_n^{-1}(\widetilde A') \cap A_n$.
Then $\widetilde A' \subseteq A_n'$ for all large $n$, and $f_n: \widetilde A_n\longrightarrow \widetilde A'$ is a degree $e$ covering.

Since 
$\Mod(B(0, R_n) - D_n) \leq \Mod(B(0, R_n) - \overline{B(0,1)}) + O(1)$, 
the Euclidean diameter of $D_n$ is bounded from below.
Therefore, after passing to a subsequence, $\widetilde A_n$ converges to an annulus $\widetilde A$ by Theorem \ref{MB}.
Hence, by Lemma \ref{DB}, $\varphi_f$ has degree $\geq e \geq 1$, which is a contradiction.
\end{proof}

The following proposition gives us control on the preimages of barycentric extensions on large scale.
We will use it in \S \ref{gl} to show the degree of the limiting map is $d$.
\begin{prop}\label{lem:ap}
Let $f_n \in \Rat_d(\C)$. Let $x_n, y_n \in \Hyp^3$ with $\dist_{\Hyp^3}(x_n, y_n) \to\infty$.
Let $a_n \in \E f_n^{-1}(x_n)$. Then after passing to a subsequence, there exists $b_n \in \E f_n^{-1}(y_n) \in \Hyp^3$ with 
$$
\lim_{n\to\infty} \dist_{\Hyp^3}(a_n, b_n)/ \dist_{\Hyp^3}(x_n, y_n) \leq 1.
$$
\end{prop}
\begin{proof}
By naturality, we assume $x_n = a_n = \bm 0$ with $A_n':=A(x_n, y_n) = B(0, R'_n) - \overline{B(0,1)}$.
Since $\E f_n(\bm 0) = \bm 0$, after passing to a subsequence, we assume $f_n \to f \in \overline{\Rat_d(\C)}$ of degree $\geq 1$ by Proposition \ref{bdd}.

Assume first that $A_n'$ contains no critical values of $f_n$. Then $f_n^{-1}(A_n')$ consists of $d$ annuli counted multiplicities.
Since $\varphi_f$ has degree $\geq 1$, there exists a sequence of components $A_n\subseteq f_n^{-1}(A_n')$ converging to an annulus $A$.

Let $(\hat a_n, \hat b_n)$ approximate $A_n$. Then it is easy to see that $\dist_{\Hyp^3}(a_n, \hat a_n)$ is bounded.
Let $M_n(\bm 0) = \hat b_n$, and $L_n(\bm 0) = y_n$.
By Proposition \ref{lem:annb}, $\dist_{\Hyp^3}(y_n, \E f_n(\hat b_n))$ is bounded.
Thus after passing to a subsequence, we assume $L_n^{-1}\circ f_n \circ M_n \to g \in \overline{\Rat_d(\C)}$ of degree $\geq 1$ by Proposition \ref{bdd}.

Applying Proposition \ref{lem:pi} to $L_n^{-1}\circ f_n \circ M_n$, there exists $R>0$ and $b_n \in \E f_n^{-1}(y_n)$ with $\dist_{\Hyp^3}(b_n, \hat b_n) \leq R$.

Hence,  $\dist_{\Hyp^3}(a_n, b_n) = \dist_{\Hyp^3}(\hat a_n, \hat b_n) + O(1) = 2\pi \Mod(A_n) + O(1)$.
Since $f_n:A_n \longrightarrow A_n'$ is a covering, $\Mod(A_n) \leq \Mod(A_n')$. 

Since $\dist_{\Hyp^3}(x_n, y_n) = 2\pi \Mod(A_n') \to \infty$,
$$
\lim_{n\to\infty} \dist_{\Hyp^3}(a_n, b_n)/ \dist_{\Hyp^3}(x_n, y_n) = \lim_{n\to\infty} \Mod(A_n)/\Mod(A_n') \leq 1.
$$

More generally, after passing to a subsequence, we can decompose
$$
[x_n, y_n] = [x_{0,n} = x_n, x_{1,n}] \cup [x_{1,n}, x_{2,n}] \cup ... \cup [x_{k-1, n}, x_{k, n} = y_n]
$$
so that each annulus $A(x_{i,n}, x_{i+1,n})$ contains no critical values of $f_n$.
Then the statement follows by induction.
\end{proof}

\section{$\R$-trees and branched coverings}\label{BCR}
In this section, we will define branched coverings between $\R$-trees.

\subsection*{$\R$-trees}
An $\R$-tree is a nonempty metric space $(T,\dist)$ such that any two points $x, y\in T$ are connected by a unique topological arc $[x,y] \subseteq T$, and every arc of $T$ is isometric to an interval in $\R$.

We say $x$ is an endpoint of $T$ if $T-\{x\}$ is connected; otherwise $x$ is an interior point. If $T - \{x\}$ has three or more components, we say $x$ is a branch point. 
The set of branch points will be denoted $B(T)$. 

We say $T$ is a {\em finite tree} if $B(T)$ is finite.
Note that we allow a finite tree to have an infinite end, so a finite tree may not be compact.
We will write $[x, y)$ and $(x, y)$ for $[x, y]$ with one or both of its endpoints removed.

\subsection*{Ends of an $\R$-tree}
A ray $\alpha$ in the $\R$-tree $T$ is a subtree isometric to $[0,\infty) \subseteq \R$.
Two rays $\alpha_1, \alpha_2$ are {\em equivalent}, denoted by $\alpha_1 \sim \alpha_2$, if $\alpha_1 \cap \alpha_2$ is still a ray.
The collection $\epsilon(T)$ of all equivalence classes of rays forms the set of {\em ends} of $T$.

We will use $\alpha$ to denote both a ray and the end it represents.
We say a sequence of points $x_i$ converges to an end $\alpha$, denoted by $x_i\to \alpha$, if for all $\beta \sim \alpha$, $x_i \in \beta$ for all sufficiently large $i$.

\subsection*{Convexity and subtrees}
A subset $S$ of $T$ is called {\em convex} if $x, y\in S \implies [x,y]\subseteq S$.
The smallest convex set containing $E\subseteq T$ is called the {\em convex hull} of $E$, and is denoted by $\hull(E)$.
More generally, we can easily extend the above definitions to subsets of $T\cup \epsilon(T)$.

Note that a subset $S\subseteq T$ is convex $\iff$ $S$ is connected $\iff$ $S$ is a subtree.
Moreover, $S$ is a {\em finite subtree} of $T$ $\iff$ $S$ is the convex hull of a finite set $E\subseteq T\cup \epsilon(T)$.

\subsection*{Branched coverings between $\R$-trees}
We now give the definition of a branched covering between $\R$-trees:
\begin{defn}
Let $f:T_1 \longrightarrow T_2$ be a continuous map between two $\R$-trees.
We say $f$ is a degree $d$ branched covering if 
there is a finite subtree $S \subseteq T_1$ such that
\begin{enumerate}
\item $S$ is nowhere dense in $T_1$, and $f(S)$ is nowhere dense in $T_2$.
\item For every $y\in T_2 - f(S)$, there are exactly $d$ preimages in $T_1$.
\item For every $x\in T_1 - S$, $f$ is a local isometric homeomorphism at $x$.
\item For every $x\in S$, and any sufficiently small neighborhood $U$ of $f(x)$, $f: V-f^{-1}(f(V\cap S)) \longrightarrow U-f(V\cap S)$ is an isometric covering\footnote{Note that the covering is isometric if $f$ is isometric restricting on each connected component of $V - f^{-1}(f(V\cap C(f)))$.}, where $V$ is the component of $f^{-1}(U)$ containing $x$.
\end{enumerate}
\end{defn}

\subsection*{Local degree and critical locus}
Let $f: T_1 \longrightarrow T_2$ be a degree $d$ branched covering, and $x\in T_1$, we define the {\em local degree} at $x$, denoted as $\deg_x(f)$ as the degree of the isometric covering of $f: V-f^{-1}(f(V\cap S)) \longrightarrow U-f(V\cap S)$ for sufficiently small neighborhood $U$ of $f(x)$.
We define 
$$
C(f) = \{x\in T_1: \deg_x(f) \geq 2\}
$$ 
as the {\em critical locus} of $f$.
Note that $C(f) \subseteq S$.

One shall think of the finite subtree $S$ in the definition as an intermediate step recording all `potential' critical points for the branched covering, while $C(f)$ is the set of critical points.
For our purposes, it is more convenient to introduce this intermediate subtree $S$.

Here are some properties which are ready to check using definitions.
\begin{prop}
Let $f: T_1 \longrightarrow T_2$ be a degree $d$ branched covering. Then
\begin{enumerate}
\item $C(f)$ is a closed set.
\item For $y\in T_2$, $\sum_{f(x) = y} \deg_x(f) = d$.
\item If $\deg_x(f) = 1$, then $f$ is a local isometric homeomorphism at $x$.
\item For every $x\in T_1$, and any sufficiently small neighborhood $U$ of $f(x)$, $f: V - f^{-1}(f(V\cap C(f))) \longrightarrow U - f(V\cap C(f))$ is an isometric covering where $V$ is the component of $f^{-1}(U)$ containing $x$.
\item If $U$ is a subtree disjoint from $C(f)$, then $f$ maps $U$ isometrically to $f(U)$.
\end{enumerate}
\end{prop}

\section{Limits of barycentric extensions on rescalings of $\Hyp^3$} \label{gl}
In this section, we will construct an $\R$-tree, called the {\em asymptotic cone}, by taking the limit of rescalings of $\Hyp^3$.
We will construct a limiting map $F$ on the asymptotic cone, and will prove Theorem \ref{DGL} and Theorem \ref{DN} by analyzing the basic properties of the map $F$.

\subsection{An equivalent condition on degeneration}
Let $f_n\in \Rat_d(\C)$. We say $f_n$ is {\em degenerating}, denoted by $f_n \to\infty$, if $f_n$ escapes every compact set of $\Rat_d(\C)$.
We first show that $f_n$ is degenerating if and only if there is `loss of mass' for the barycentric extension.

\begin{prop}\label{LossOfMass}
Let $f_n\in \Rat_d(\C)$, and $r_n:=\max_{y\in \E f_n^{-1}(\bm 0)} \dist_{\Hyp^3}(\bm 0, y)$. Then $f_n \to\infty$ if and only if $r_n \to \infty$.
\end{prop}
\begin{proof}
We first prove the `if' direction. Assume $r_n \to \infty$, and suppose for contradiction that $f_n\not\to\infty$. Let $M_n\in \PSL_2(\C)$ such that $\E f_n(M_n(\bm 0)) = \bm 0$ and $\dist_{\Hyp^3}(\bm 0, M_n(\bm 0)) \to \infty$. 

After passing to a subsequence, we may assume that $f_n\to f\in \Rat_d(\C)$ and $M_n\to M\in \overline{\Rat_1(\C)}$ with $\deg(\varphi_M) = 0$.
By Lemma \ref{WeakLimit}, $(f_n\circ M_n)_*\mu_{S^2}$ converges to the delta measure $(f\circ \varphi_M)_*\mu_{S^2}$ which is a contradiction to $(f_n\circ M_n)_*\mu_{S^2}$ is balanced.

Conversely, assume $f_n\to\infty$. Suppose for contradiction that $r_n\not\to\infty$. 
After passing to a subsequence, we may assume $r_n\to r$,  $f_n\to f \in \overline{\Rat_d(\C)}$ with $\deg(\varphi_f) < d$. We may also assume the set of critical values of $f_n$ converges to a finite set $V$. 

Let $A'$ be an annulus of modulus $K > 0$ compactly contained in $\hat \C - V$. 
We assume $C \notin \overline{A'}$ if $\varphi_f \equiv C$.
Then for all sufficiently large $n$, $A'$ contains no critical values of $f_n$, so $f_n^{-1}(A')$ consists of $d$ annuli counted multiplicities.

Since $f_n \to \infty$, the set of holes $\mathcal{H}(f)\neq \emptyset$.
Let $a\in \mathcal{H}(f)$.
By Lemma \ref{lem:ns}, after passing to a subsequence, there exists a sequence of components $A_n \subseteq f_n^{-1}(A')$ converging to $a$.

Let $M_n \in \PSL_2(\C)$ so that $M_n^{-1} (A_n)\subseteq\C$ separates $0$ and $\infty$ with the diameter of the bounded component of $\C - M_n^{-1} (A_n)$ equal to $1$.
Since $\Mod(M_n^{-1}(A_n)) \geq K/d$, 
after passing to a subsequence, $M_n^{-1}(A_n)$ converges to an annulus $A$ by Theorem \ref{MB}.

After passing to a subsequence, we may assume
$$
f_n\circ M_n: M_n^{-1}(A_n) \longrightarrow A'
$$ 
is a degree $e$ covering map, and $f_n \circ M_n\to g \in \overline{\Rat_d(\C)}$.
By Lemma \ref{DB}, $\deg(\varphi_g) \geq e \geq 1$.

Applying Proposition \ref{lem:pi} to $f_n \circ M_n$, we conclude that there exists $R>0$ and a sequence $x_n \in (\E f_n \circ M_n)^{-1}(\bm 0)$, with $\dist_{\Hyp^3}(\bm 0, x_n) \leq R$.
Note that $M_n(x_n) \in \E f_n^{-1}(\bm 0)$.
Since $A_n$ converges to $a$, $M_n\to\infty$. 
Therefore, $\dist_{\Hyp^3}(\bm 0, M_n (x_n)) \to \infty$, so $r_n \to\infty$, which is a contradiction.
\end{proof}

A similar proof yields the following, which gives a criteria for a point $a\in \hat\C$ to be a hole.
\begin{lem}\label{lem:CH}
Let $f_n\in \Rat_d(\C)$ converge to $f\in \overline{\Rat_d(\C)}$ of degree $\geq 1$.
Let $a\in \mathcal{H}(f) \subseteq \hat\C$ and $y_n \in \Hyp^3$ be a sequence. After passing to a subsequence, there exists a sequence $x_n \in \E f_n^{-1}(y_n)$ converging to $a$.
\end{lem}
\begin{proof}
Let $L_n \in \PSL_2(\C)$ so that $L_n(\bm 0) = y_n$.
Let $V_n$ be the set of critical values of $f_n$.
After passing to a subsequence, we may assume $L_n \to L \in \overline{\Rat_1(\C)}$, and $L_n^{-1}(V_n)$ converges to a finite set $V$.
Let $A'$ be an annulus with $\Mod(A) = K > 0$ compactly contained in $\hat \C - V - \mathcal{H}(L)$.

Since $\overline{A'}\cap\mathcal{H}(L)=\emptyset$, $L_n(A')$ converges to $L(A')$, which is a point if $\deg(L) = 0$ or an annulus if $\deg(L) = 1$.

Since $\overline{A'} \cap V = \emptyset$, for sufficiently large $n$, $(L_n^{-1}\circ f_n)^{-1}(A') = f_n^{-1}(L_n(A'))$ consists of $d$ annuli counted multiplicities.
By Lemma \ref{lem:ns}, after passing to a subsequence, there exists a sequence of components $A_n \subseteq f_n^{-1}(L_n(A'))$ converging to the hole $a$.

Let $M_n \in \PSL_2(\C)$ so that $M_n^{-1} (A_n)\subseteq\C$ separates $0$ and $\infty$ with the diameter of the bounded component of $\C - M_n^{-1} (A_n)$ equal to $1$.
By the same argument as in Proposition \ref{LossOfMass}, after passing to a subsequence, $L_n^{-1}\circ f_n \circ M_n \to g \in \overline{\Rat_d(\C)}$ of degree $\geq 1$. 

Applying Proposition \ref{lem:pi} to $L_n^{-1}\circ f_n \circ M_n$, there exists $R>0$ and a sequence $x_n' \in (L_n^{-1} \circ \E f_n \circ M_n)^{-1}(\bm 0)$, with 
$\dist_{\Hyp^3}(\bm 0, x_n') \leq R$.
Let $x_n := M_n(x_n') \in \E f_n^{-1}(y_n)$.
Since $A_n\to a$, $M_n(\bm 0) \to a$. Hence, $x_n \to a$.
\end{proof}

\subsection{Ultralimits and asymptotic cones of metric spaces.}\label{sec:ulac}
Proposition \ref{LossOfMass} suggests the natural objects to consider are the rescalings of metric spaces.
In the following, we will construct asymptotic cones of $\Hyp^3$.
We refer the readers to \cite{Roe03,KapovichLeeb95} for details.

\subsection*{Ultrafilters}
We begin by briefly reviewing the theory of ultrafilters on $\N$. 
A subset $\omega\subseteq \powerset(\N)$ of the power set of $\N$ is called an {\em ultrafilter} if it satisfies the following 4 properties:
\begin{enumerate}
\item If $A, B\in \omega$, then $A\cap B\in \omega$;
\item If $A\in \omega$ and $A\subseteq B$, then $B\in \omega$;
\item $\emptyset\notin \omega$;
\item If $A\subseteq \N$, then either $A \in \omega$ or $\N-A\in \omega$.
\end{enumerate}

By virtue of the $4$ properties, one can think of an ultrafilter $\omega$ as defining a {\em finitely additive $\{0,1\}$-valued probability measure} on $\N$. 
The sets of measure $1$ are precisely those belonging to the ultrafilter $\omega$. 

We will call a set in $\omega$ as {\em $\omega$-big} or simply {\em big}. Its complement is called {\em $\omega$-small} or simply {\em small}.
If a specific property is satisfied by a $\omega$-big set, we will say this property holds $\omega$-almost surely.

\begin{example}
Let $a\in \N$. We define
$$
\omega_a:=\{A\subseteq \N: a\in A\}.
$$
Then it can be easily verified that $\omega_a$ is an ultrafilter on $\N$.
\end{example}

An ultrafilter of the above type will be called a {\em principal ultrafilter}. It can be shown that an ultrafilter is principal if and only if it contains a finite set.
An ultrafilter that is not principal is called a {\em non-principal ultrafilter}.
The existence of a non-principal ultrafilter is guaranteed by Zorn's lemma.

Let $\omega$ be a non-principal ultrafilter on $\N$.
Let $x_n$ be a sequence in a metric space $(X,\dist)$ and $x\in X$. 
We say $x$ is the {\em $\omega$-limit} of $x_n$, denoted by 
$$
\lim_\omega x_n = x,
$$
if for every $\epsilon>0$, the set $\{n: \dist(x_n, x) < \epsilon\}$ is big.

It can be easily verified (see \cite{KapovichLeeb95}) that 
\begin{enumerate}
\item If the $\omega$-limit of $x_n$ exists, then it is unique.
\item If $x_n$ is contained in a compact set, then the $\omega$-limit exists.
\item If $x = \lim_{n\to\infty} x_n$ in the standard sense, then $x = \lim_{\omega} x_n$.
\item If $x = \lim_{\omega} x_n$, then there exists a subsequence $n_k$ such that $x = \lim_{k\to\infty} x_{n_k}$ in the standard sense.
\end{enumerate}

From these properties, one should intuitively think (as one of the benefits) of the non-principal ultrafilter $\omega$ as performing all the subsequence-selection in advance. 
The non-principal ultrafilter $\omega$ allows all sequences in compact spaces to converge automatically, without the need to pass to a subsequences.

\begin{conv}
From now on and throughout the rest of the paper, we will fix a non-principal ultrafilter $\omega$ on $\N$. 
The statements shall be interpreted as $\omega$-almost surely without explicit mention.
\end{conv}

\subsection*{Ultralimits of pointed metric spaces.}
Let $(X_n, p_n, \dist_n)$ be a sequence of pointed metric spaces with basepoint $p_n$. Let $\mathcal{X}$ denote the set of sequences $(x_n) \subseteq X_n$ such that $\dist_n(x_n, p_n)$ is a bounded function of $n$. We define an equivalence relation $\sim$ on $\mathcal{X}$ by
$$
(x_n) \sim (y_n) \Leftrightarrow \lim_\omega \dist_n(x_n, y_n) = 0.
$$
We use $[(x_n)]$ to denote the equivalence class of the sequence $(x_n)$.

Let $X_\omega = \mathcal{X} / \sim$. We define
$$
\dist_\omega ([(x_n)], [(y_n)]) := \lim_\omega \dist_n(x_n, y_n).
$$

The function $\dist_\omega$ makes $X_\omega$ a metric space, and is called the {\em ultralimit} of $(X_n, p_n, \dist_n)$ with respect to the ultrafilter $\omega$.
It is written as $\lim_\omega (X_n, p_n, \dist_n)$ or simply $\lim_\omega (X_n, p_n)$ for short.

The ultralimit of $X_n$ has many of the desired properties (see \cite[\S 7.5]{Roe03} and \cite{KapovichLeeb95} for associated definitions and proofs):
\begin{enumerate}
\item The ultralimit is always a complete metric space.
\item The ultralimit of a sequence of {\em length spaces} is a length space.
\item The ultralimit of a sequence of {\em geodesic spaces} is a geodesic space. 
\item If $X_n$ is {\em proper} and $(X_n, p_n) \to (Y, y)$ in the sense of Gromov-Hausdorff, then $(Y, y)$ is canonically isometric to $ \lim_\omega (X_n, p_n)$.
\end{enumerate}

\subsection*{Asymptotic cones of $\Hyp^3$}
Given a positive sequence $r_n\to \infty$, we define the {\em asymptotic cone} 
$({^r\Hyp^3}, x^0, \dist)$ with rescaling $r_n$ as the ultralimit
$$
({^r\Hyp^3}, x^0, \dist) = \lim_\omega (\Hyp^3, \bm 0, \frac{1}{r_n}\dist_{\Hyp^3}).
$$
From the definition, a point $x\in {^r\Hyp^3}$ is represented by a sequence $x_n \in \Hyp^3$ with $\lim_\omega \frac{\dist_{\Hyp^3}(\bm 0, x_n)}{r_n} < \infty$, and $[(x_n)] = [(y_n)]\iff \lim_\omega \frac{\dist_{\Hyp^3}(x_n, y_n)}{r_n} = 0$.

It is well known that ${^r\Hyp^3}$ is an $\R$-tree that has uncountably many branches at every point (see \cite{Roe03,KapovichLeeb95}).
In the sequel \cite{L19}, we will show that ${^r\Hyp^3}$ is canonically isometric to the Berkovich hyperbolic space of the complexified Robinson's field.



Let $z\in \hat\C \cong S^2$. 
We denote $\gamma(t, z)\in \Hyp^3$ as the point of distance $t$ from $\bm 0$ in the direction of $z$.
Given any sequence $z_n \in \hat\C$, the ray
$$
s(t) = [(\gamma(t\cdot r_n, z_n))]
$$
is a geodesic ray parameterized by arc length in ${^r\Hyp^3}$.
Thus, $(z_n)$ corresponds to an end $\alpha \in \epsilon({^r\Hyp^3})$. We denote this by
$$
z_n \twoheadrightarrow_\omega \alpha.
$$

\subsection{The construction of the limiting map $F$ on ${^r\Hyp^3}$}
We first show that Theorem \ref{RationalLip} allows us to construct a natural limiting map on the asymptotic cone.
\begin{prop}\label{GLTree}
Let $f_n\in \Rat_d(\C)$, and $r_n \to \infty$. If there exists $K>0$ such that $\dist_{\Hyp^3}(\bm 0, \E f_n(\bm 0)) \leq K r_n$, then the natural limiting map
\begin{align*}
F = \lim_\omega \E f_n: {^r\Hyp^3} &\longrightarrow {^r\Hyp^3}\\
[(x_n)] &\mapsto [(\E f_n(x_n))]
\end{align*}
is well-defined and uniformly Lipschitz.
\end{prop}
\begin{proof}
We first show $(\E f_n(x_n))$ represents a point in ${^r\Hyp^3}$.
Since $[(x_n)] \in {^r\Hyp^3}$, there exists $M>0$ so that 
$$
\dist_{\Hyp^3}(x_n, \bm 0) \leq M r_n.
$$ 
Hence Theorem \ref{RationalLip} implies that
$$
\dist_{\Hyp^3}(\E f_n(\bm 0), \E f_n(x_n)) \leq Cd\cdot M r_n.
$$
By our assumption, $\dist_{\Hyp^3}(\bm 0, \E f_n(\bm 0)) \leq K r_n$. Thus,
\begin{align*}
\dist_{\Hyp^3}(\bm 0, \E f_n(x_n)) &\leq \dist_{\Hyp^3}(\bm 0, \E f_n(\bm 0)) + \dist_{\Hyp^3}(\E f_n(\bm 0), \E f_n(x_n))\\
&\leq (K+Cd\cdot M)r_n.
\end{align*}
Therefore, $(\E f_n(x_n))$ represents a point in ${^r\Hyp^3}$.

We now show the definition $F$ does not depend on the choice of representatives.
If $[(x_n)] = [(y_n)]\in {^r\Hyp^3}$, then 
$$
\lim_{\omega} \dist_{\Hyp^3}(x_n, y_n) / r_n = 0.
$$
By Theorem \ref{RationalLip},
$$
\lim_{\omega} \dist_{\Hyp^3}(\E f_n(x_n), \E f_n(y_n)) / r_n \leq \lim_{\omega}Cd\cdot \dist_{\Hyp^3}(x_n, y_n)/r_n = 0,
$$
so $[(\E f_n(x_n))] = [(\E f_n(y_n))]$.

Thus, $F$ is well-defined. By the same computation, $F$ is $Cd$-Lipschitz.
\end{proof}

If $f_n\to\infty$ in $\Rat_d$, then $r_n := \max_{y\in \E f_n ^{-1}(\bm 0)} \dist_{\Hyp^3}(\bm 0, y) \to \infty$ by Proposition \ref{LossOfMass}. For $y_n\in \E f_n^{-1}(\bm 0)$, Theorem \ref{RationalLip} gives 
$$
\dist_{\Hyp^3}(\bm 0, \E f_n(\bm 0)) \leq Cd \cdot \dist_{\Hyp^3}(\bm 0, y_n) \leq Cd \cdot r_n.
$$
Hence, we have the following immediate corollary.

\begin{cor}\label{CGL}
Let $f_n \to\infty$ in $\Rat_d(\C)$, and 
$$
r_n := \max_{y\in \E f_n ^{-1}(\bm 0)} \dist_{\Hyp^3}(\bm 0, y).
$$
Then the natural limiting map
\begin{align*}
F = \lim_\omega \E f_n: {^r\Hyp^3} &\longrightarrow {^r\Hyp^3}\\
[(x_n)] &\mapsto [(\E f_n(x_n))]
\end{align*}
is well-defined and uniformly Lipschitz.
\end{cor}

\subsection{The limiting map $F$ is dynamically natural}
Let $f_n \to \infty$ in $\Rat_d(\C)$ with $r_n = \max_{y\in \E f_n ^{-1}(\bm 0)} \dist_{\Hyp^3}(\bm 0, y)$. 
Let $F : {^r\Hyp^3} \longrightarrow {^r\Hyp^3}$ be the limiting map.
In the following, we will verify $F$ is dynamically natural.

We record the following lemma, which follows immediately from the naturality of the barycentric extension and Proposition \ref{bdd}.
\begin{lem}\label{lem:bddu}
Let $f_n \in \Rat_d(\C)$. Then $\dist_{\Hyp^3}(L_n(\bm 0), \E f_n(M_n(\bm 0)))$ is bounded $\omega$-almost surely if and only if $\deg (\lim_\omega L_n^{-1}\circ f_n\circ M_n) \geq 1$.
\end{lem}

For barycentric extensions, $\E f\circ \E g$ is usually different from $\E (f\circ g)$.
We now show this discrepancy disappears when we pass to the limit.
\begin{proof}[Proof of Theorem \ref{DN}]
Given $x\in {^r\Hyp^3}$, we will show $F^N(x) = \lim_\omega \E (f_n^N)(x)$.

Let $M_{0,n} \in \PSL_2(\C)$ with $x = [(M_{0,n}(\bm 0))]$.
Let $M_{i,n}\in\PSL_2(\C)$ with $M_{i,n}(\bm 0) = \E f_n (M_{i-1,n}(\bm 0))$.
Note that $F^N(x) = [(M_{N,n}(\bm 0))]$.

Since $\E f_n \circ M_{i-1,n}(\bm 0) = M_{i,n}(\bm 0)$, by Lemma \ref{lem:bddu}, $\lim_{\omega} M_{i,n}^{-1} \circ f_n \circ M_{i-1,n}$ has degree $\geq 1$, for $i=1,..., N$.
Hence $\lim_{\omega} M_{N,n}^{-1} \circ f_n^N \circ M_{0,n}$ has degree $\geq 1$ by Lemma \ref{lem:avoh}. Thus, 
$$
\dist_{\Hyp^3}(M_{N,n}(\bm 0), \E (f_n^N) (M_{0,n}(\bm 0)))
$$ 
is bounded $\omega$-almost surely by Lemma \ref{lem:bddu}.
Therefore, 
$$
\lim_\omega \E (f_n^N)(x) = [(\E (f_n^N) (M_{0,n}(\bm 0)))] = [(M_{N,n}(\bm 0))] = F^N(x).
$$
\end{proof}

\subsection{The limiting map $F$ is a branched covering}
In the following, we will verify that $F$ is a branched covering between $\R$-trees.

Let $c_{1,n},..., c_{2d-2,n}$ be the $2d-2$ critical points of $f_n$.
Let $\kappa_1,..., \kappa_{2d-2}$ be the corresponding ends of ${^r\Hyp^3}$, i.e., $c_{i,n}\twoheadrightarrow_\omega \kappa_i$ for all $i =1,..., 2d-2$.
We define the {\em critical subtree} by 
$$
S = \hull (\kappa_1,..., \kappa_{2d-2})\subseteq {^r\Hyp^3}.
$$

Let $f_n(c_{1,n}),..., f_n(c_{2d-2,n})$ be the $2d-2$ critical values of $f_n$.
The corresponding ends are called the {\em ends of critical values}.

The following lemma follows from the analysis of the dynamics $f_n$ on annuli with large modulus.
\begin{lem}\label{EP}
Let $x, y \in {^r\Hyp^3}$. 
Assume the projections of the ends of critical values onto $[x,y]$ are not interior points. Then there are segments 
$$[a_1,b_1],..., [a_k, b_k] \subseteq {^r\Hyp^3}$$ with disjoint interiors\footnote{Note that $a_i$ (or $b_i$) may equal to $a_j$ (or $b_j$) for distinct $i,j$.} such that
\begin{enumerate}
\item $F$ maps $[a_i,b_i]$ linearly onto $[x,y]$ with derivative $e_i \in \Z_{\geq 1}$; and
\item $\sum_{i=1}^k e_i = d$; and
\item $F^{-1}([x,y]) = \bigcup_{i=1}^k [a_i, b_i]$.
\end{enumerate}
\end{lem}
\begin{proof}

Since the projections of the ends of critical values onto $[x,y]$ are not interior points, we can choose representatives $(x_n), (y_n)$ of $x, y$ so that the annulus $A_n:=A(x_n, y_n)$\footnote{\label{note1}See page~\pageref{defA} for the definition.} contains no critical values of $f_n$. 


Let $f_n^{-1}(A_n) = \{A_{1,n},..., A_{k,n}\}$,
and $e_i$ be the degree\footnote{Note that the number of components $k$ and the degree $e_i$ are defined $\omega$-almost surely.} of the covering $f_n: A_{i,n} \longrightarrow A_n$. 

Let $(a_{i,n}, b_{i,n})$ approximate $A_{i,n}$\footref{note1}.
By Proposition \ref{lem:annb}, $\dist_{\Hyp^3}(x_n, \E f_n(a_{i,n}))$ is bounded.
By Proposition \ref{lem:ap}, there exists $u_n \in \E f_n^{-1}(\bm 0)$ with 
$$
\lim_{\omega} \dist_{\Hyp^3}(u_n, a_{i,n})/ \dist_{\Hyp^3}(\bm 0, \E f_n(a_{i,n})) \leq 1.
$$
Thus, $\lim_{\omega} \dist_{\Hyp^3}(u_n, a_{i,n})/r_n \leq \lim_{\omega} \dist_{\Hyp^3}(\bm 0, \E f_n(a_{i,n}))/r_n  = \dist (x^0, x)$.

By our choice of $r_n$, $\lim_{\omega}\dist_{\Hyp^3}(\bm 0, u_n)/r_n \leq 1$.
Thus, $(a_{i,n})$ represents a point in ${^r\Hyp^3}$.
Let $a_i:=[(a_{i,n})]$.
Similarly, let $b_i=[(b_{i,n})]$.

By construction, $F(a_i) = x$ and $F(b_i) = y$. Moreover, 
\begin{align*}
\dist (a_i, b_i) &= \lim_\omega r_n^{-1} \cdot (2\pi \Mod A_{i,n} + O(1))\\
&= \lim_\omega r_n^{-1} \cdot (2\pi \Mod A_{n}/e_i + O(1))\\
&= \lim_\omega r_n^{-1} \cdot (\dist_{\Hyp^3} (x_n, y_n)/e_i  + O(1)) = \dist(x,y)/e_i.
\end{align*}

Since the same arguement works for any pair of points in $[x,y]$, we conclude that $[a_i,b_i]$ maps linearly onto $[x,y]$ with derivative $e_i\in \Z_{\geq 1}$.

Since the components $A_{1,n},..., A_{k,n}$ are disjoint, $[a_1, b_1],..., [a_k, b_k]$ have disjoint interior.
Since the degree of $f_n$ is $d$, $\sum_{i=1}^k e_i = d$.

To prove the last property, it suffices to show $F^{-1}(x) = \{a_1,..., a_k\}$.
By naturality, it suffices to show $x^0 \in \{a_1,..., a_k\}$ assuming $\E f_n(\bm 0) = \bm 0$.
By Lemma \ref{lem:bddu}, $\deg(\lim_\omega f_n) \geq 1$.
Therefore, $A_{i,n}$ converges to an annulus $A_i$ for some $i$. 
So $x^0 \in \{a_1,..., a_k\}$.
\end{proof}

Similar to Lemma \ref{EP}, we have
\begin{lem}\label{IsoU}
Let $x, y \in {^r\Hyp^3}$. 
If the projections of the critical ends onto $[x,y]$ are not interior points, 
then $F$ maps $[x,y]$ linearly onto $[F(x), F(y)]$ with derivative $e\in \Z_{\geq 1}$.

Moreover, if $[x,y] \subseteq {^r\Hyp^3 - S}$, then $F$ maps $[x,y]$ onto $[F(x), F(y)]$ isometrically.
\end{lem}
\begin{proof}
The first part follows from a similar proof as in Lemma \ref{EP}.

If $[x,y] \subseteq {^r\Hyp^3 - S}$, then the projections of the critical ends onto $[x,y]$ is a single point $a\in [x,y]$.
To prove the moreover part, we consider two cases. 

If $a=x$ (or $a=y$), then we can choose representatives $x_n, y_n$ so that all the critical points of $f_n$ are contained in the component of $\hat\C - A(x_n, y_n)$ corresponding to $x$ (or $y$ respectively). Then the degree of the associated covering map is $1$. Therefore, $F$ is an isometry.

If $a\in (x,y)$, then by the above arguement, $F$ is an isometry on $[x,a]$ and $[a,y]$.
Thus, it suffices to show the image of $[x, a)$ and $(a, y]$ are disjoint.

Let $x_n, y_n$ represent $x, y$ respectively. 
Let $a_n$ be in the geodesic segment connecting $x_n, y_n$ such that $a = [(a_n)]$.
By naturality, we may assume that $a_n = \bm 0$ with $\E f_n(\bm 0) = \bm 0$, and $x_n, y_n$ are contained in the geodesic with ends $0, \infty \in \hat\C$.

Let $f=\lim_\omega f_n$.
By Lemma \ref{lem:bddu}, $\deg(\varphi_f) \geq 1$.
If $\deg(\varphi_f) > 1$, then there are at least two distinct limits of critical points of $f_n$ in $\hat\C$, contradicting $[x,y] \subseteq {^r\Hyp^3 - S}$. Thus $\deg(\varphi_f) = 1$

Modifying $a_n$ if necessary, we may assume that all the critical points converge to $1\in \hat\C$. So by Lemma \ref{lem:ns}, $\mathcal{H}(f) = \{1\}$. Thus, $f_n$ converges to $\varphi_f$ uniformly near $0, \infty$ by Lemma \ref{comcon}.
Since $\deg(\varphi_f) = 1$, $\varphi_f(0) \neq \varphi_f(\infty)$, so the image of $[x, a)$ and $(a, y]$ are disjoint.
\end{proof}

\begin{lem}\label{M1}
Let $y \in {^r\Hyp^3} - F(S)$. Then it has exactly $d$ preimages.
\end{lem}
\begin{proof}
Since $y\in {^r\Hyp^3} - F(S)$, we can find two points $u, v \in {^r\Hyp^3}$ so that $y\in (u,v) \subseteq [u,v] \subseteq {^r\Hyp^3} - F(S)$.

By Lemma \ref{EP}, there are segments $[a_1,b_1],..., [a_k, b_k]$ such that $F$ maps $[a_i, b_i]$ to $[x,y]$ with derivative $e_i$ and $\sum_{i=1}^k e_i = d$. 
Since $[a_i, b_i] \in {^r\Hyp^3} - S$, by Lemma \ref{IsoU}, $e_i = 1$, so $k=d$.
Since $[a_i, b_i]$ have disjoint interiors, $y$ has exactly $d$ preimages.
\end{proof}

The {\em tangent space} $T_x{^r\Hyp^3}$ at $x$ is defined as the set of components of ${^r\Hyp^3}-\{x\}$.
Let $a\in T_x{^r\Hyp^3}$ be a tangent direction with the corresponding component $S_a$.
For all $y\in S_a \subseteq {^r\Hyp^3}-\{x\}$ sufficiently close to $x$, the projections of the critical ends onto $[x,y]$ are not interior points.
Thus by Lemma \ref{IsoU}, $F$ induces a map on the tangent space 
$$
DF_x: T_x{^r\Hyp^3} \longrightarrow T_{F(x)}{^r\Hyp^3}.
$$

We define the local degree of $F$ at the direction $a\in T_x{^r\Hyp^3}$ by $\deg_{x, a} F = e$, where $e$ is the local derivative in the direction of $a$ in Lemma \ref{IsoU}.
We define the local degree of $F$ at $x$ as
$$
\deg_x F := \sum_{\{a: DF_x(a) = b\}} \deg_{x, a} F,
$$
where $b\in T_{F(x)}{^r\Hyp^3}$.
Using a similar argument as Lemma \ref{EP}, it is not hard to check that the definition is independent of the choice of $b$ and every point has exactly $d$ preimages counted multiplicities.

We remark that it follows from Lemma \ref{IC} below that this definition of local degrees agrees with the definition in \S \ref{BCR}.
Once the connection with the Berkovich dynamics is established in the sequel \cite{L19}, the tangent spaces of ${^r\Hyp^3}$ are naturally isomorphic to $\Proj^1_K$ for some field $K$. 
The tangent map $DF_x$ is a rational map on $\Proj^1_K$, and $\deg_x F$ is the degree of the rational map.

\begin{lem}\label{IC}
Let $x \in {^r\Hyp^3}$. Then for sufficiently small neighborhood $U$ of $F(x)$, 
$$
F: V-F^{-1}(F(V\cap S)) \longrightarrow U - F(V\cap S)
$$ 
is an isometric covering of degree $\deg_{x} F$ where $V$ is the component of $F^{-1}(U)$ containing $x$.
\end{lem}
\begin{proof}

We choose a neighborhood $U$ of $F(x)$ so that the component $V$ of $F^{-1} (U)$ containing $x$ contains no other preimages of $F(x)$.
Let $y \in U - F(S)$, we will show that $y$ has exactly $\deg_x F$ preimages in $V$. 

Suppose that the projections of the ends of critical values onto $[F(x),y]$ are not interior points.
Since $V$ contains no other preimages of $F(x)$, 
$$
F^{-1}([F(x),y]) \cap V = \bigcup_{i=1}^{k}[x, y_i]
$$ 
by Lemma \ref{EP}.
By definition of the local degree, $y$ has $\deg_x F$ preimages in $V$ counted multiplicities.
By Lemma \ref{M1}, each $y_i$ has multiplicity $1$. So $k = \deg_x F$ and $y$ has exactly $\deg_x F$ preimages in $V$.

More generally, we can decompose
$$
[F(x), y] = [x_0=F(x), x_1] \cup ... \cup [x_{n-1}, x_n = y]
$$ 
into finitely many segments where $x_i$ are the projections of the ends of critical values on $[F(x), y]$.
Then by induction, each $x_i$ has $\deg_x F$ preimages in $V$ counted multiplicities. 
By Lemma \ref{M1}, each preimage of $y$ has multiplicity $1$.
Therefore, $y$ has exactly $\deg_x F$ preimages in $V$.

Since each component of $V-F^{-1}(F(V\cap S))$ is contained in ${^r\Hyp^3 - S}$, by Lemma \ref{IsoU}, the covering map is isometric.
\end{proof}

We are ready to prove Theorem \ref{DGL}:
\begin{proof}[Proof of Theorem \ref{DGL}]
By Corollary \ref{CGL}, we have a limiting map $F: {^r\Hyp^3} \longrightarrow {^r\Hyp^3}$.
It is easy to check that $S$ and $F(S)$ are both nowhere dense in ${^r\Hyp^3}$.
So by Lemma \ref{M1} and Lemma \ref{IC}, $F$ is a branched covering of degree $d$.
\end{proof}

\subsection{Degenerating sequences in $M_d$}
In the following, we will give a version of Theorem \ref{DGL} in the {\em moduli space of rational maps}
$$
M_d = \Rat_d(\C)/\PSL_2(\C).
$$

A sequence $[f_n]\in M_d$ is {\em degenerating}, denoted by $[f_n]\to\infty$, if $[f_n]$ escapes every compact set of $M_d$.
Equivalently, $[f_n] \to\infty$ if and only if {\em every} sequence of representatives $f_n \in \Rat_d$ is degenerating.

For $[f] \in M_d$, we define
$$
r([f]):= \inf_{x\in \Hyp^3}\max_{y\in \E f^{-1}(x)} \dist_{\Hyp^3}(x, y).\footnote{We do not know if the infimum is achieved.
Consider the function $R(x):=\max_{y\in \E f^{-1}(x)} \dist_{\Hyp^3}(y, x)$.
It is not hard to show that $R(x) \to \infty$ as $x\to \infty$ in $\Hyp^3$.
We do not know if $R(x)$ is continuous.}
$$

It follows immediately from the naturality and Proposition \ref{LossOfMass} that:
\begin{prop}\label{LossOfMassC}
$[f_n]\to \infty$ in $M_d$ if and only if $r([f_n]) \to \infty$.
\end{prop}

Thus, the natural scale to consider is $r([f_n])$, and we have
\begin{theorem}\label{DGLC}
Let $[f_n] \to \infty$ in $M_d$ and $r_n:= r([f_n])$.
Let $f_n$ be a representative of $[f_n]$ with
$$
\max_{y\in \E f_n^{-1}(\bm{0})} \dist_{\Hyp^3}(\bm{0}, y) \leq r_n+1.
$$
Then the corresponding limiting map
$$
F = \lim_\omega \E f_n: {^r\Hyp^3} \longrightarrow {^r\Hyp^3}
$$
is a branched covering of degree $d$ with no totally invariant point.
\end{theorem}
\begin{proof}
By Theorem \ref{DGL}, it suffices to show there are no totally invariant points.
Since $r_n:= r([f_n])$, for any point $x\in {^r\Hyp^3}$,
$$
\max_{y\in F^{-1} (x)} \dist(x, y) \geq 1.
$$
Therefore, there are no totally invariant points.
\end{proof}

\subsection{Geometric Limits}
In the following, we will illustrate how to extract a geometric limit of $\E f_n$ in the standard sense.

Let $F: T \longrightarrow T$ be a map on an $\R$-tree $T$ and $F_n: \Hyp^3 \longrightarrow \Hyp^3$.
Following the definitions in \cite[\S 11]{McM09}, we say $F_n$ converges {\em geometrically} to $F$ if there exists a sequence of maps
$$
h_n : T \longrightarrow \Hyp^3
$$
and a sequence of rescalings $r_n\to \infty$ such that
\begin{enumerate}
\item {\em Rescaling:} We have
$$
\dist(x,y) = \lim_{n\to\infty} r_n^{-1} \dist_{\Hyp^3}(h_n(x), h_n(y))
$$
for all $x,y\in T$;
\item {\em Conjugacy:} For all $x\in T$, we have
$$
\lim_{n\to\infty} r_n^{-1} \dist_{\Hyp^3}(h_n(F(x)), F_n(h_n(x))) = 0.
$$
\end{enumerate}


\begin{theorem}\label{thm:gl}
Let $f_n \to\infty$ in $\Rat_d(\C)$, with limiting map $F: {^r\Hyp^3} \rightarrow {^r\Hyp^3}$.
Let $T = \hull(A) \subseteq {^r\Hyp^3}$ be an invariant subtree where $A$ is countable.
After passing to a subsequence, $\E f_n$ converges geometrically to $F: T \longrightarrow T$.
\end{theorem}
\begin{proof}
Enumerate $A = \{x_0, x_1, x_2,...\}$ and let $T_k = \hull(\{x_0,x_1,..., x_k\}) \subseteq {^r\Hyp^3}$.
Then $T_1 \subseteq T_2 \subseteq ...$ is an increasing sequence of finite trees. 
Note
$$
T = \bigcup_{k=1}^\infty T_k,
$$
as $\bigcup_{k=1}^\infty T_k$ is the smallest connected subset containing $A$.

Since each $T_k$ is a finite tree, it is separable, i.e., it contains a countable dense subset.
Since $T = \bigcup_{k=1}^\infty T_k$, $T$ is also separable.
Thus by adding countably many points, we may assume $A$ is dense in $T$.

By induction on $k$, we will construct $h_n$ on $T_k$ with:
\begin{enumerate}
\item $x= [(h_n(x))]$ for all $x\in T_k$; and
\item $\dist_{\Hyp^3}(h_n(x), h_n(y)) \leq 2r_n \cdot \dist(x, y)$ for all $x, y\in T_k$.
\end{enumerate}

As the base case, we choose representatives $x_0=[(x_{0,n})], x_1=[(x_{1,n})]$ so that
$$
\dist_{\Hyp^3}(x_{0,n}, x_{1,n}) \leq 2r_n \cdot \dist(x_0, x_1).
$$ 
Define
$$
h_n: T_1=[x_0, x_1] \longrightarrow \Hyp^3
$$
by $h_n(x_0) = x_{0,n}$, $h_n(x_1) = x_{1,n}$, and extending linearly between $[x_0, x_1]$ and the geodesic segment connecting $h_n(x_0)$ and $h_n(x_1)$.
By construction, the two properties are satisfied.

Suppose $h_n: T_k \longrightarrow \Hyp^3$ is constructed.
If $T_{k+1} = T_k$, then we do nothing.

Otherwise, $T_{k+1} = T_k \cup [a,x_{k+1}]$, where $a$ is the projection of $x_{k+1}$ onto $T_k$.
Choose a representative $x_{k+1}=[(x_{k+1,n})]$ with 
$$
\dist_{\Hyp^3}(x_{k+1,n}, h_n(a)) \leq 2r_n \cdot \dist(x_{k+1}, a).
$$
Define the map $h_n$ on $[a,x_{k+1}]$ by $h_n(x_{k+1}) = x_{k+1,n}$, and extending linearly between $[a,x_{k+1}]$ and the geodesic segment connecting $h_n(a)$ and $h_n(x_{k+1})$.
By induction hypothesis and triangle inequality, $h_n: T_{k+1} \longrightarrow \Hyp^3$ satisfies the two properties.

Let $h_n: T=\bigcup_{k=1}^\infty T_k \longrightarrow \Hyp^3$.
By property (2), $h_n$ is $2r_n$-Lipschitz, so it is continuous.
By property (1) and the definition of $F$, for $x, y\in T$,
$$
\dist(x,y) = \lim_\omega r_n^{-1} \dist_{\Hyp^3}(h_n(x), h_n(y)),
$$
and
$$
\lim_\omega r_n^{-1} \dist_{\Hyp^3}(h_n(F(x)), \E f_n(h_n(x))) = 0.
$$

Since the ultralimit of a sequence of real numbers is in the limit set, we can find a subsequence so that for all $x, y\in \{x_0,...,x_k\}$,
$$
\dist(x,y) = \lim_{n\to\infty} r_{n}^{-1} \dist_{\Hyp^3}(h_{n}(x), h_{n}(y)),
$$
and
$$
\lim_{n\to\infty} r_{n}^{-1} \dist_{\Hyp^3}(h_{n}(F(x)), \E f_{n}(h_{n}(x))) = 0.
$$

Using a standard diagonal argument, we can pass to a further subsequence so that the rescaling and conjugacy conditions hold for all $x, y\in A$.

Since $h_n$ is $2$-Lipschitz with respect to the metric $r_n^{-1}\dist_{\Hyp^3}$ in the target, and $F$, $\E f_n$ are $Cd$-Lipschitz with a universal constant $C$, the rescaling and conjugacy conditions hold for all $x, y\in T = \overline{A}$.
\end{proof}

\begin{remark}
Let $f_n$ be a degenerating sequence of Blaschke product, with limiting map $F: {^r\Hyp^3} \rightarrow {^r\Hyp^3}$.
Let $A$ be the grand orbit of the base point $x^0$, and $T = \hull(A) \subseteq {^r\Hyp^3} $.
Then it can be checked that Theorem \ref{thm:gl} recovers the geometric limit constructed in \cite[Theorem 1.3]{McM09}.
\end{remark}


\section{Length spectra and translation lengths}\label{PL}
In this section, we will show that we can recover the limiting length spectra from the translation lengths for the limiting map, and prove Theorem \ref{TL}.

Let $f_n \to \infty$ in $\Rat_d(\C)$ with 
$$
r_n:= \max_{y\in \E f_n^{-1}(\bm{0})} \dist_{\Hyp^3}(\bm{0}, y).
$$ 
Let $F: {^r\Hyp^3} \longrightarrow {^r\Hyp^3}$ be the limiting map.




Let $\alpha: [0,\infty) \longrightarrow {^r\Hyp^3}$ represent an end.
By Lemma \ref{IsoU}, for all large $t$, $F$ is either an isometry or expanding with constant derivative $e\in \Z_{\geq 2}$ on $\alpha([t, \infty))$.

If $F$ is eventually an isometry, then $\dist(\alpha(t), x^0) - \dist(F(\alpha(t)), x^0)$ is eventually constant.
Otherwise, $\dist(\alpha(t), x^0) - \dist(F(\alpha(t)), x^0) \to -\infty$.
Therefore, the {\em translation length}
$$
L(\alpha, F) = \lim_{x_i \to \alpha} \dist(x_i, x^0) - \dist(F(x_i), x^0)
$$
is well defined, which can possibly be $-\infty$. 

We are ready to prove Theorem \ref{TL}.
\begin{proof}[Proof of Theorem \ref{TL}]
We consider the case for $q=1$. By Theorem \ref{DN}, the general case can be proved by taking iterations.\footnote{It can be verified from the argument below that the period of $\mathcal{C}$ is $q$.}

Let $z_n$ be a fixed point for $f_n$, with $z_n \twoheadrightarrow_\omega \alpha$.
By naturality, we assume $z_n = 0$.

First, assume $\alpha$ is not critical.
Let $M_{t,n}(z) = e^{-t\cdot r_n}z$. 
Then $\alpha(t) = [(M_{t,n}(\bm 0))]$ represents the end.
Let $L_{t,n}(\bm 0) = \E f_n \circ M_{t,n} (\bm 0)$ with $L_{t,n}(0) = 0$. To simplify the notations, we define
$$
f_{t,n}:=L_{t,n}^{-1} \circ f_n \circ M_{t,n}
$$
and 
$$
f_t := \lim_\omega L_{t,n}^{-1} \circ f_n \circ M_{t,n} \in \overline{\Rat_d(\C)}.
$$ 
Note $\deg(\varphi_{f_t}) \geq 1$ by Lemma \ref{lem:bddu}.
Since $\alpha$ is not critical, $\deg(\varphi_{f_t}) = 1$ for all large $t$.

Since $F^{-1}(x^0)$ is finite, for all large $t$, $\alpha(t)$ separates the end of $\alpha$ from $F^{-1}(x^0)$.
Therefore, for all large $t$
$$
\E f_{t,n}^{-1}(L_{t,n}(\bm 0)) = M_{t,n}^{-1}(\E f_n^{-1}(\bm 0)) \to \infty \in \hat\C.
$$
Thus, by applying Lemma \ref{lem:CH} to $f_{t,n}$ and $L_{t,n}(\bm 0)$, $0$ is not a hole.
Hence, $f_{t,n}$ converges uniformly to $f_t$ near $0$ by Lemma \ref{comcon}. This gives $f_n(0) = 0 \twoheadrightarrow_\omega F(\alpha)$, so $\alpha$ is a fixed end.
Moreover, for all sufficiently large $t$,
$$
\lim_\omega (L_{t,n}^{-1} \circ f_n \circ M_{t,n})'(0) = \varphi_{f_t}'(0) \neq 0.
$$
By the chain rule, we get
$$
\lim_\omega r_n^{-1}(-\log |L_{t,n}'(0)| + \log |f_n'(0)| + \log |M_{t,n}'(0)|) = \lim_\omega r_n^{-1} \log |\varphi_{f_t}'(0)| = 0.
$$
Therefore, for all large $t$, we have
\begin{align*}
L(\alpha, F) &= \dist(\alpha(t), x^0) - \dist(F(\alpha(t)), x^0) \\
&= -\lim_\omega r_n^{-1} \log |M_{t,n}'(0)| + \lim_\omega r_n^{-1} \log |L_{t,n}'(0)|\\
&= \lim_\omega r_n^{-1} \log |f_n'(0)|.
\end{align*}

If $\alpha$ is critical, then a similar argument shows $\lim_{\omega} \frac{\log |(f_n)'(0)|}{r_n} = -\infty$.
\end{proof}


\appendix

\section{The barycentric extension of $z^2$}\label{BP}
In this section, we will compute the extension of a degree $2$ rational map on $S^2$.
Since $\Isom\Hyp^3 \times \Isom\Hyp^3$ acts transitively on $\Rat_2(\C)$, by naturality, there is only one map $z^2$ to study.

Let $(r, \theta, h)$ be the cylindrical coordinate system for the hyperbolic $3$-space $\Hyp^3$.
More precisely, $(r, \theta)$ is the polar coordinates for the projection of $p\in \Hyp^3$ onto the hyperbolic plane $P \subseteq \Hyp^3$ and $h$ is the signed hyperbolic distance of $p$ to the plane.
We will prove
\begin{theorem}\label{Degree2Coordinate}
Let $f(z) = z^2$.
In the cylindrical coordinate, there exists a real analytic function $\delta: [0,\infty) \longrightarrow \R$ with 
\begin{itemize}
\item $\delta(0) = 0$,
\item $\delta(r) > 0$ when $r>0$,
\item $\lim_{r\to\infty}\delta(r)= 0$,
\end{itemize}
such that the barycentric extension $\E f$ is given by
$$
\E f(r, \theta, h) = (\log (\cosh(r)) - \delta(r), 2\theta, 2h).
$$
\end{theorem}

\begin{proof}
Since the barycentric extension is natural and $e^{2t} f (z/e^t) = f(z)$, $\E f$ sends the hyperbolic plane $\Hyp^2_t$ of $h = t$ to $\Hyp^2_{2t}$. Hence, $\E f$ preserves the hyperbolic plane $\Hyp^2_0$ of $h=0$. By naturality, the restriction $\E f:\Hyp^2_t \longrightarrow \Hyp^2_{2t}$ is the same as $\E f|_{\Hyp^2_0}$. Therefore, it suffices to show
\begin{align*}
\E f|_{\Hyp^2_0}: \Hyp^2_0 &\longrightarrow \Hyp^2_0 \\
(r,\theta) &\mapsto (\log (\cosh(r)) - \delta(r), 2\theta).
\end{align*}
Since $f(z)$ is invariant under complex conjugation, by naturality, 
$$
\E f|_{\Hyp^2_0} (r,0) = (g(r), 0).
$$ 
A priori, $g(r)$ can be negative. We first show $g(r) > 0$ when $r>0$.

Let $u = \begin{bmatrix} t \\ 0 \\ 0 \end{bmatrix} \in \Hyp^3 \subseteq \R^3$.
For $0<t<1$, the $x$-coordinate of the integral
$$
\int_{S^2} \zeta d(f\circ M_{u})_*\mu_{S^2}(\zeta) = \int_{S^2} f(\zeta) (\frac{1-\| u\|^2}{\|\zeta - u\|^2})^2 d\mu_{S^2}(\zeta)\in \R^3
$$
is
\begin{align*}
&\;\frac{(1-t^2)^2}{4\pi(1+t^2)^2}\int_0^\pi \int_0^{2\pi}\frac{h(\varphi)\cos(2\theta)}{(1-\frac{2t}{1+t^2}\sin\varphi\cos\theta)^2} d\theta \sin\varphi d\varphi\\
= &\; \frac{(1-t^2)^2}{4\pi(1+t^2)^2}\int_0^\pi \int_0^{2\pi}h(\varphi)\cos(2\theta)\sum_{n=1}^\infty n (\frac{2t}{1+t^2}\sin\varphi\cos\theta)^{n-1}d\theta \sin\varphi d\varphi > 0.
\end{align*}
Here $h(\varphi) \geq 0$ is some non-zero function of $\varphi$.
The last inequality holds as the integral 
$$
\int_0^{2\pi}\cos(2\theta) (\cos\theta)^{n-1} d\theta \begin{cases} =0 &\mbox{if $n-1$ is odd} \\ 
>0 & \mbox{if $n-1$ is even} \end{cases}.
$$

Thus, the barycenter of $(f\circ M_{u})_*\mu_{S^2}$ is on the positive $x$-axis.
Hence $g(r) > 0$ when $r>0$.

Since $f(z) = e^{i2\theta} f (z/ e^{i\theta})$, by naturality, 
$$
\E f|_{\Hyp^2_0} (r,\theta) = R_{2\theta}(\E f|_{\Hyp^2_0} (r, 0)) = (g(r), 2\theta),
$$ 
where $R_{2\theta}$ is rotation by $2\theta$.

Therefore, it suffices to compute $g(r)$.
We need the following lemma:

\begin{lem}\label{ComputeDegree2}
Let $f_t(z) = z\frac{z-t}{1-tz}$ for $t\in (0,1)$. Then there exists a positive function $\kappa:(0,1)\longrightarrow \R_+$ such that $\E f_t(\bm 0) = (\kappa(t), \pi, 0)$ for all $t\in (0,1)$.
\end{lem}

\begin{figure}[h]
\centering
\includegraphics[width=0.5\textwidth]{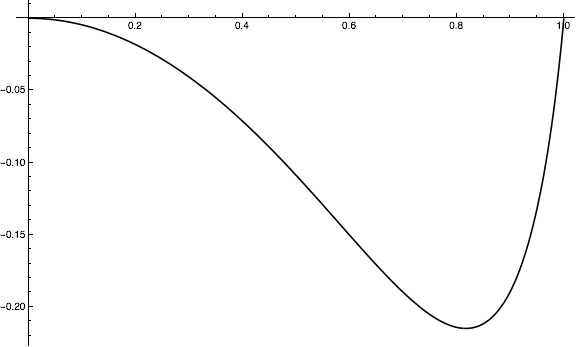}
\caption{The graph of the function $I_t$.}
\end{figure}

\begin{proof}[Proof of Lemma \ref{ComputeDegree2}]
Let 
\begin{align*}
P: \C &\longrightarrow \C\times \R \cong \R^3 \\
z &\mapsto (\frac{2z}{1+|z|^2}, \frac{|z|^2-1}{|z|^2+1})
\end{align*}
be the stereographic projection.
Note that $\E f_t(\bm 0) = \bm 0$ if and only if
$$
\int_{S^2} P (f_t (P^{-1}(\zeta))) d\mu_{S^2}(\zeta) = \vec{0} \in \R^3 \cong \C\times \R.
$$
By symmetry, the second component in $\C\times \R$ of the integral is always $0$.
Changing the variables to $z$, the first component of the integral equals to
$$
I_t:= \int_\C \frac{2z\frac{z-t}{1-tz}}{1+|z\frac{z-t}{1-tz}|^2}\frac{4}{(1+|z|^2)^2} \frac{i}{2}|dz|^2.
$$
In polar coordinate, we may express the integral as
\begin{align*}
I_t &= \int_0^\infty \int_0^{2\pi} \frac{2z\frac{z-t}{1-tz}}{1+|z\frac{z-t}{1-tz}|^2}\frac{4}{(1+|z|^2)^2} rd\theta dr\\
&= \int_0^\infty (\int_{\partial B_r} \frac{2z\frac{z-t}{1-tz}}{1+|z\frac{z-t}{1-tz}|^2}\frac{4}{(1+|z|^2)^2} 
\frac{|z|dz}{iz}) dr
\end{align*}
where $B_r$ is the disk centered at $0$ of radius $r$.

Let 
$J_t(r)$
be the inner integral. 
We will show that $J_t(r)$ is negative for all $r \neq 1$. Note that on $\partial B_r$, $\bar z = r^2/z$, so we have
\begin{align*}
J_t(r) &= \int_{\partial B_r} \frac{2z\frac{z-t}{1-tz}}{1+|z\frac{z-t}{1-tz}|^2}\frac{4}{(1+|z|^2)^2} \frac{|z|dz}{iz}\\
&= \int_{\partial B_r} \frac{2z\frac{z-t}{1-tz}}{1+r^2\frac{(z-t)(\bar z-t)}{(1-tz)(1-t\bar z)}}\frac{4}{(1+r^2)^2}
\frac{rdz}{iz}\\
&= \frac{8r}{i(1+r^2)^2}\int_{\partial B_r} \frac{\frac{z-t}{1-tz}}{1+r^2\frac{(z-t)(r^2/z-t)}{(1-tz)(1-tr^2/z)}}dz\\
&= \frac{8r}{i(1+r^2)^2}\int_{\partial B_r} \frac{(z-t)(z-tr^2)}{(1-tz)(z-tr^2)+r^2(z-t)(r^2-tz)}dz\\
&= \frac{16r\pi}{(1+r^2)^2} \Res_{z\in B_r} \frac{(z-t)(z-tr^2)}{(1-tz)(z-tr^2)+r^2(z-t)(r^2-tz)}.
\end{align*}

Let $F(z) = (z-t)(z-tr^2)$ and $G(z) = (1-tz)(z-tr^2)+r^2(z-t)(r^2-tz)$. 
Note $G(r) = r(1-tr)^2+ r^3(r-t)^2 > 0$
as $t\in (0,1)$ and $r>0$. 
Since the coefficient of $z^2$ in $G(z)$ is negative, $G(z)$ has two real roots $x_1, x_2$ with $x_1< r < x_2$.

Note that $G(t) = t(1-t^2)(1-r^2)$ and $G(tr^2) = tr^4(r^2-1)(1-t^2)$ have different signs when $r\neq 1$,
and $G(\max\{t, tr^2\}) > 0$, so $x_1$ lies between $t$ and $tr^2$. 
Hence, $F(x_1) <0$.
Therefore, when $r\neq 1$,
$$
\Res_{z\in B_r} \frac{(z-t)(z-tr^2)}{(1-tz)(z-tr^2)+r^2(z-t)(r^2-tz)} = \frac{F(x_1)}{-t(1+r^2)(x_1-x_2)} <0.
$$

Hence, $J_t(r) < 0 $ for all $r\neq 1$, so $I_t < 0$.
So 
$$
\int_{S^2} P (f_t (P^{-1}(\zeta))) d\mu_{S^2}(\zeta) = \begin{bmatrix} -v(t) \\ 0 \\ 0 \end{bmatrix} \in \R^3
$$
for some positive function $v(t)>0$. Thus, the barycenter of $(f_t)_* \mu_{S^2}$ is on the negative $x$-axis, so
$$
\E f_t(\bm 0) =(\kappa(t), \pi, 0)
$$
for some positive function $\kappa(t) > 0$.
\end{proof}

Let $M_t(z) = \frac{z-t}{1-tz}$. By comparing zeros, $M_{t^2} \circ f \circ M_{-t}^{-1}(z) = f_{\frac{2t}{1+t^2}}(z)$. 
Note that $\dist_{\Hyp^3}(\bm 0, M_t(\bm 0)) = \log \frac{1+t}{1-t}$.
Hence, by Lemma \ref{ComputeDegree2} and naturality,
$$
\E f|_{\Hyp^2_0}(\log \frac{1+t}{1-t}, 0) = \E f|_{\Hyp^2_0}(\log \frac{1+t}{1-t}, \pi) = (\log \frac{1+t^2}{1-t^2} - \tilde \delta(t), 0),
$$
where $\tilde \delta(t) > 0$ for $t\in (0,1)$.

Thus, by substituting $r = \log \frac{1+t}{1-t}$, there exists $\delta(r)$ with $\delta(r) >0$ for $r>0$ so that
$$
\E f|_{\Hyp^2_0}(r, 0) = (\log (\cosh(r)) - \delta(r), 0).
$$
Since both $\E f|_{\Hyp^2_0}$ and $\log (\cosh(r))$ are real analytic, $\delta(r)$ is real analytic.

Since the spherical derivative of $f$ at $1 \in \hat\C \cong S^2$ is $2$, 
\begin{align*}
\log 2 &= \lim_{r\to\infty} \dist_{\Hyp^3} ((r,0,0), \bm 0)-\dist_{\Hyp^3}(\E f(r,0,0), \bm 0)\\
&= \lim_{r\to\infty} (r- (\log (\cosh(r)) -\delta(r))).
\end{align*}
Since
$\lim_{r\to\infty} (r-\log(\cosh(r))) = \log 2$, $\lim_{r\to\infty}\delta(r)= 0$.
\end{proof}

\section{The uniform bound for barycentric extensions in higher dimensions}\label{HD}
In this section, we explain how our method can be generalized to give uniform bounds of barycentric extensions of proper quasiregular maps in higher dimensions.
\begin{theorem}\label{Lip}
Let $f:S^n \longrightarrow S^n$ be a proper $K$-quasiregular map of degree $d$. Then the norm of the derivative of its barycentric extension $\E f: \Hyp^{n+1} \longrightarrow \Hyp^{n+1}$ satisfies
$$
\sup_{x\in \Hyp^{n+1}} \| \Derivative \E f (x)\| \leq C\cdot (Kd)^{\frac{1}{n-1}}.
$$
Here the norm is computed with respect to the hyperbolic metric, and $C = C(n)$ depends only on the dimension $n$.
\end{theorem}

The proof of the above theorem is almost identical to the proof of Theorem \ref{RationalLip} once we have all the ingredients.
We refer the readers to \cite{rickman_1993,G61,G62,Z67} for detailed discussions on quasiconformal and quasiregular maps in higher dimensions.

\subsection*{Conformal capacity and separating set}
The theory of extremal length and extremal width extends naturally to higher dimensions.
For our purpose, we consider an equivalent formulation of {\em capacity} of a {\em condenser}.

In the literature, a {\em condenser} in $\R^n$ is a pair $(A, C)$ where $A$ is open in $\R^n$ and $C\neq \emptyset$ is a compact subset of $A$.

Let $U = A- C, E_1 = \partial C$ and $E_2 = \partial A - E_1$.
To make our discussion parallel with \S \ref{sec:ELW}, we shall also call the triple $(U, E_1, E_2)$ arising in this way a condenser in $\R^n$.

\begin{defn}
The {\em (conformal) capacity} of $(U, E_1, E_2)$ is defined by
$$
\capacity_U(E_1, E_2) := \inf_u \int_U |\nabla u|^n dm
$$
where $dm$ is the Lebesgue measure on $\R^n$, and the infimum is taken over all nonnegative functions $u \in C^0\cap W^1_{n,loc}$\footnote{$W^1_{n}$ consists of all real valued functions $u\in L^n$ with weak first order partial derivatives which are themselves in $L^n$. $W^1_{n,loc}$ consists those functions that are locally in $W^1_n$.} 
with compact support and $u|_{E_1} \geq 1$.
The {\em extremal distance} is defined by 
$$
\Mod_U(E_1, E_2) = \frac{1}{\capacity_U(E_1, E_2)}.
$$
\end{defn}

Let $(U, E_1, E_2)$ be a condenser with $U$ bounded. We say a closed set $\sigma$ is {\em separating} if $\sigma \subseteq U$ and if there are two (relative) open sets $V_1$ and $V_2$ of $\overline U$ with $E_1 \subseteq V_1$, $E_2 \subseteq V_2$ and $\overline U - \sigma = V_1 \cup V_2$.
The following is a higher dimensional generalization of extremal width of family of separating curves.

\begin{defn}
Let $\Sigma^{sep}$ be the set of all separating sets of a bounded condenser $(U, E_1, E_2)$ in $\R^n$. The {\em extremal width of separating sets}\footnote{It is usually referred to the {\em modulus of separating sets} in the literature. We use the terminology extremal width to make the discussion parallel with \S \ref{sec:ELW}.} is defined by
$$
\EW(\Sigma^{sep}) := \inf_{f\wedge \Sigma^{sep}} \int_{\R^n} f^{\frac{n}{n-1}} dm.
$$
Here, $f\wedge \Sigma^{sep}$ means $f$ is a nonnegative Borel function on $\R^n$ such that
$$
\int_{\sigma} f d\mathcal{H}^{n-1} \geq 1 \text{ for all } \sigma \in \Sigma^{sep},
$$
where $d\mathcal{H}^{n-1}$ is the $n-1$-dimensional Hausdorff measure.
\end{defn}

The following theorem is proved in \cite{Z67,G62}, and is a direct analogue of Theorem \ref{thm:eds}.
\begin{theorem}\label{M=C}
Let $(U, E_1, E_2)$ be a bounded condenser in $\R^n$ and $\Sigma^{sep}$ be the set of all separating sets. 
Then 
$$
\Mod_U(E_1, E_2) = \EW(\Sigma^{sep})^{n-1}.
$$
\end{theorem}

\subsection*{Quasiregular maps}
A mapping $f:\Omega\longrightarrow \R^n$ of a domain $\Omega \subseteq \R^n$ is {\em $K$-quasiregular} if
\begin{enumerate}
\item $f \in C^0\cap W^1_{n,loc}$; and
\item there exists $K$, $1\leq K < \infty$ such that
$$
|f'(x)|^n \leq K \Jac f(x) \text{ a.e.}.
$$
\end{enumerate}

Quasiregular maps share many nice properties with holomorphic functions.
For example, a non-constant quasiregular map is discrete and open (see \cite[Chapter I \S 4]{rickman_1993}).

The following theorem, proved in \cite[Chapter II \S 1 and \S 10]{rickman_1993}, is an analogue of Theorem \ref{thm:tr}.
\begin{theorem}\label{KO}
Let $(U, E_1, E_2)$ and $(U', E_1', E_2')$ be two condensers.
Let $f: \overline{U} \longrightarrow \overline{U'}$ be a proper quasiregular map of degree $d$ with $f(E_1) = E_1'$ and $f(E_2) = E_2'$. Then
$$
\Mod_{U'}(E_1', E_2') \leq Kd \cdot \Mod_{U}(E_1, E_2).
$$
\end{theorem}

\begin{remark}
The theorem in \cite{rickman_1993} is stated for connected $U$.
It is ready to verify that the proof of the theorem does not require this assumption.
\end{remark}

\subsection*{Proof of Theorem \ref{Lip}}
Now that we have all the ingredients, the proof is almost identical to the proof of Theorem \ref{RationalLip}.
Using Theorem \ref{M=C} and Theorem \ref{KO} as substitutes for Theorem \ref{thm:eds} and Theorem \ref{thm:tr}, we can show the following analogue of Proposition \ref{prop:rl}.
\begin{prop}\label{BY}
Let $f:S^n \longrightarrow S^n$ be a proper $K$-quasiregular map of degree $d$ with $\E f(\bm 0) = \bm 0$. Then
$$
\|F_{y} (\bm 0, \bm 0) ^{-1}\| \leq M (Kd)^{\frac{1}{n-1}}
$$
for some constant $M = M(n)$ depending only on the dimension $n$.
\end{prop}

Theorem \ref{Lip} now follows immediately from Proposition \ref{BY}.


\end{document}